\documentclass[twoside,reqno]{amsart}
\usepackage{amsmath}
\usepackage{amsfonts}
\usepackage{amsthm}
\usepackage{amssymb}
\usepackage{graphicx,psfrag,epsfig}
\usepackage{color}
\usepackage{mathrsfs}
\usepackage{pdfsync}
\usepackage{pst-all}
\usepackage{cite}
\usepackage{times}

\newcommand{\ra}{\rangle}
\newcommand{\la}{\langle}

\newcommand{\R}{\mathbb{R}}

\newcommand{\N}{\mathbb{N}}

\newcommand{\mC}{\mathcal{C}}

\newcommand{\Om}{\Omega}
\newcommand{\OmT}{\Omega_T}

\newcommand{\rd}{\mathrm{d}}

\newcommand{\bqn}{\begin{equation}}
\newcommand{\eqn}{\end{equation}}
\newcommand{\bqnn}{\begin{equation*}}
\newcommand{\eqnn}{\end{equation*}}
\newcommand{\bear}{\begin{eqnarray}}
\newcommand{\eear}{\end{eqnarray}}
\newcommand{\bean}{\begin{eqnarray*}}
\newcommand{\eean}{\end{eqnarray*}}
\newcommand{\esxyp}{\end{split}}
\newcommand{\bsp}{\begin{split}}
\newcommand{\balg}{\begin{align}}
\newcommand{\ealg}{\end{align}}
\newcommand{\balgg}{\begin{align*}}
\newcommand{\ealgg}{\end{align*}}

\newcommand{\ol}{\overline}

\newcommand{\calphazwei}{\mathcal{C}^{1+\frac{\alpha}{2},2+\alpha}}

\newcommand{\alphah}{{(\frac{\alpha}{2},\alpha)}}

\newcommand{\alphaeinsh}{{(\frac{1+\alpha}{2})}}

\newcommand{\dhr}{\mathrel{\lhook\joinrel\relbar\kern-.8ex\joinrel\lhook\joinrel\rightarrow}}

\newcommand{\wt}{\widetilde}
\newcommand{\us}{\underset}

\newcommand{\sml}{\sum\limits}

\newtheorem{theorem}{Theorem}[section]
\newtheorem{corollary}[theorem]{Corollary}

\newtheorem{proposition}[theorem]{Proposition}

\newtheorem{remark}[theorem]{Remark}
\numberwithin{equation}{section}
\begin{document}
\title{On Global Solutions for Quasilinear One-Dimensional Parabolic Problems with Dynamical Boundary Conditions}
%\thanks{}
\author{Simon Gvelesiani}
\address{Leibniz Universit\"at Hannover, Institut f\" ur Angewandte Mathematik, Welfengarten 1, D--30167 Hannover, Germany} 
\email{gvelesiani@ifam.uni-hannover.de}
\author{Friedrich Lippoth}
\address{Leibniz Universit\"at Hannover, Institut f\" ur Angewandte Mathematik, Welfengarten 1, D--30167 Hannover, Germany} 
\email{lippoth@ifam.uni-hannover.de}
\author{Christoph Walker}
\address{Leibniz Universit\"at Hannover, Institut f\" ur Angewandte Mathematik, Welfengarten 1, D--30167 Hannover, Germany} 
\email{walker@ifam.uni-hannover.de}
\date{\today}
%
%%%%%%%%%%%%%%%%%%%%%%%%%%%%%%%
\begin{abstract}
We provide sufficient and almost optimal conditions for global existence of classical solutions in parabolic H\"older spaces to quasilinear one-dimensional parabolic problems with dynamical boundary conditions. 

\end{abstract}
%%%%%%%%%%%%%%%%%%%%%%%%%%%%%%%
%
\maketitle
\pagestyle{myheadings}
\markboth{\sc{S. Gvelesiani, F. Lippoth \& Ch. Walker}}{\sc{One-dimensional parabolic problems with dynamical boundary conditions}}
%
%
%%%%%%%%%%%%%%%%%%%%%%%%%%%%%%%
%%%%%%%%%%%%%%%%%%%%%%%%%%%%%%%
\section{Introduction} \label{sec:int}
%%%%%%%%%%%%%%%%%%%%%%%%%%%%%%%
%%%%%%%%%%%%%%%%%%%%%%%%%%%%%%%

The focus of the present research is on the one-dimensional quasilinear parabolic equation
\begin{equation}\label{1}
u_t-a(t,x,u,u_x)u_{xx}=f(t,x,u,u_x)\ ,\qquad (t,x)\in (0,T)\times (-\ell,\ell)\ ,
\end{equation}
supplemented with a nonlinear dynamical boundary condition
\begin{equation}\label{2}
u_t\pm b(t,x,u,u_x)u_{x}=g(t,x,u,u_x)\ ,\qquad t\in (0,T)\, ,\quad x=\pm\ell\ ,
\end{equation}
and the initial condition $u(0,\cdot)=u^0$ on $[-\ell,\ell]$, where $T>0$ and $\ell>0$ are given fixed numbers.\footnote{Here and in the following, equations at the boundary points $x=\pm\ell$ as \eqref{2} involving $\pm$ signs (or $\mp$ signs) are to be understood as two equations with a $+$ sign (or $-$ sign) at $x=+\ell$ and a $-$ sign (or $+$ sign) at $x=-\ell$. The sign convention is chosen such that $\pm u_x(\pm \ell)$ represents the ''outward'' normal derivative.} Equations of the form \eqref{1}, \eqref{2} occur in various fields of natural sciences, we refer e.g. to \cite{Baz1,Con,Esc1,Gol} and the references therein. In the past decades many different aspects of problems with dynamical boundary conditions (also in higher space dimensions) have been investigated by means of different techniques (e.g. \cite{Baz2,Biz,Biz1,Con,Esc1, Gal,Mey,Vasy,DPZ, Gui}) for well-posedness issues, also related to the sign of the function $b$ in \eqref{2} (e.g. \cite{Bel1,
 Biz,Tem,Vaz}) or for possibly degenerate equations (see \cite{Gal} and the references therein). Research has, of course, also focused on questions regarding global existence and related {\it a priori} estimates and blow up phenomena \cite{Bel1,Bel2,Con,Ell,Esc1,Esc2,Fil1,Fil2,Fil3,Mey,Pet}. None of these lists of references is complete though.

The starting point of our investigations are, on the one hand, the results of \cite{Con} for equation \eqref{1} subject to the gradient-independent dynamical boundary condition
\begin{equation}\label{2a}
u_t\pm b(t,x,u)u_{x}=g(t,x,u)\ ,\qquad t\in (0,T)\, ,\quad x=\pm\ell\ ,
\end{equation}
and, on the other hand, the results of \cite{Ter} related to \eqref{1} but subject to more standard boundary conditions of Dirichlet, Neuman, or Robin type. In \cite{Con} criteria were found for the existence of global (i.e. existing on the whole time interval $[0,T)$) classical solutions for the boundary value problem \eqref{1}, \eqref{2a}, though not restricted to one dimension. Roughly speaking, it was shown that if $a=b$ and if the growth of the right-hand side $f(t,x,u,p)$ in the gradient variable $p$ is not faster than $1+\vert p\vert^{1+\alpha}$, then bounded solutions are global provided that $\alpha=0$ in the general quasilinear case and $\alpha\in [0,1)$ in the semilinear case (i.e. if $a=b$ is also independent of $u$). Furthermore, it was shown in \cite{Con} that if $\alpha>1$, then gradients of bounded solutions may blow up in finite time and thus, solutions do not exist globally in general. At least for the case of Dirichlet or Neuman boundary conditions it is we
 ll known that a quadratic growth of the quotient $f(t,x,u,p)/a(t,x,u,p)$ as $p\rightarrow \infty$ is an almost optimal condition for global {\it a priori} estimates on the gradient of solutions to \eqref{1}, \eqref{2a}. Actually, global {\it a priori} estimates can be derived under  the Bernstein-Nagumo condition
\bqn\label{237}
\Big| \frac{f(t,x,u,p)}{a(t,x,u,p)}  \Big| \leq \psi(|p|)
\eqn
with a positive function $\psi$ obeying
\bqn\label{238}
\int\limits_{}^{\infty}\frac{\rho\,\rd \rho}{\psi(\rho)} = \infty\,, 
\eqn
in which case $\psi$ can grow even faster than quadratic. This last conditions is sharp in the sense that its violation leads in several situation to a gradient blow up, see \cite{Con}, \cite{Qui}, the introduction of \cite{Ter}, and the references given therein. Nevertheless, in the one-dimensional case with Dirichlet or Neuman boundary conditions, the Bernstein-Nagumo condition was improved in \cite{Ter} by means of the doubling of variables technique to cope with right-hand sides $f=f_1+f_2$, where only $f_1$ satisfies \eqref{237} and $f_2$ enjoys some monotonicity properties. The long time behavior of solutions also depends on the boundary conditions \cite{Qui}.

The aim of the present paper is to prove similar results as just described on existence and {\it a priori} estimates for the case of \eqref{1} subject to the
nonlinear dynamical boundary condition \eqref{2} or \eqref{2a}. More precisely, we first prove with Theorem~\ref{ExLTheorem1} an existence result for \eqref{1} subject to \eqref{2} in  parabolic H\"older spaces (recalled in the Appendix~\ref{app}) which includes a general global existence criterion that can be simplified when restricting to the gradient independent boundary condition \eqref{2a}. In Section~\ref{sec:globsol} we then show that under the Bernstein-Nagumo condition one may derive $L_\infty$-estimates on the gradient of bounded classical solutions to \eqref{1} subject to \eqref{2}, see Theorem~\ref{A1}. A similar result is obtained in Theorem~\ref{Theorem_11} for a weaker version of the Bernstein-Nagumo condition and in Theorem~\ref{Theorem_12} for mixed boundary conditions, that is, when a dynamic boundary condition is imposed on one boundary point and e.g. a Dirichlet condition on the other. These gradient estimates can be used to derive H\"older estimates on th
 e gradient (see Corollary~\ref{Theorem_15}) which, for the special case of \eqref{1} subject to \eqref{2a}, imply that bounded solutions exist globally in time, see Corollary~\ref{c3}. Finally, in Proposition~\ref{Theorem_14b} we provide conditions under which solutions to \eqref{1} subject to \eqref{2a} are bounded and thus exist globally in time as stated in Corollary~\ref{C33}.

%%%%%%%%%%%%%%%%%%%%%%%%%%%%%%%

%%%%%%%%%%%%%%%%%%%%%%%%%%%%%%%
\section{{Local Solutions}} \label{sec:locsol}
%%%%%%%%%%%%%%%%%%%%%%%%%%%%%%%
%%%%%%%%%%%%%%%%%%%%%%%%%%%%%%%

In this section we show a local existence result in the parabolic H\"older space $\calphazwei(\Om_\tau)$ (see the Appendix for a definition and properties) with $\tau>0$ and $\Omega_\tau:= (0,\tau)\times (-\ell,\ell)$  for equation \eqref{1} subject to the dynamical boundary condition \eqref{2}. To do so, we first consider the corresponding linear problem
\begin{align}
u_t-a(t,x)u_{xx}&=f(t,x), \qquad (t,x)\in (0,T)\times (-\ell,\ell), \label{ExL7}\\
u_t\pm b(t,x)u_x&=g(t,x), \qquad\phantom{bla} t\in (0,T), \quad x=\pm \ell,\label{ExL8}\\
u(0,x)&=u^0(x), \qquad\quad\;\;\;\, x\in [-\ell,\ell] \label{ExL9}
\end{align}
and  recall the following result:

\begin{proposition}[\hspace{-.01mm}{{\bf\cite[Theorem 1.1]{Biz}}}]\label{ExLin1}
Let $\alpha\in (0,1)$. Let $a\in \mC^{\frac{\alpha}{2}, \alpha}(\Om_T)$ and $b(\cdot,\pm\ell) \in C^{\frac{1+\alpha}{2}}([0,T])$ satisfy
\bqn\label{62}
m^{-1}\leq a(t,x) \leq m, \quad  (t,x)\in \Om_T, \qquad  m^{-1} \leq b(t,\pm \ell) \leq m, \quad t\in (0,T),
\eqn
for some $m>0$ and let $f\in \mC^{\frac{\alpha}{2}, \alpha}(\Om_T)$ and $g(\cdot,\pm\ell)\in C^{\frac{1+\alpha}{2}}([0,T])$. If $u^0\in C^{2+\alpha}([-\ell,\ell])$ satisfies the compatibility condition
\begin{equation}\label{141}
\begin{split}
\big(au_{xx}^0+f\big)(0,\pm \ell) = \big(\mp bu_{x}^0+g\big)(0,\pm \ell),
\end{split}
\end{equation}
then there is a unique solution $u\in \mC^{1+\frac{\alpha}{2}, 2+\alpha}(\OmT)$ to \eqref{ExL7}--\eqref{ExL9} such that $u_t(\cdot,\pm\ell)\in C^{\frac{1+\alpha}{2}}([0,T])$. Moreover, there is a constant
\bqnn
c_l:=c_l\left(T,m, |a|^{\alphah}_{\OmT},|b(\cdot,\pm\ell)|^{\alphaeinsh}_{[0,T]}\right),
\eqnn
such that
\bqn\label{ExLin2}
|u|_{\OmT}^{(1+\frac{\alpha}{2}, 2+\alpha)}+|u_t(\cdot,\pm\ell)|_{[0,T]}^{(\frac{1+\alpha}{2})} \leq c_l\big(|f|_{\OmT}^{(\frac{\alpha}{2},\alpha)}+
|u^0|^{(2+\alpha)}_{[-\ell,\ell]}+|g(\cdot,\pm\ell)_{[0,T]}^{(\frac{1+\alpha}{2})}\big).
\eqn
\end{proposition}

Proposition~\ref{ExLin1} is stated and proved in  \cite[Theorem 1.1]{Biz} in dimensions higher than one, a detailed proof for the one-dimensional case may also be found in \cite{GveliDiss}. A more general situation is considered in \cite{DPZ} in an $L_p$-setting. The constant $c_l$ is increasing with respect to $T$. Also note that the compatibility condition for the initial value is natural as we consider classical solutions. \\

Based on the previous result we next prove an existence result for classical solutions to the nonlinear problem \eqref{1}, \eqref{2} by means of the contraction mapping principle. This naturally yields a global existence criterion that we shall exploit further in the subsequent section.  Let us point out that other methods and other solution spaces are possible as well, of course, see e.g. \cite{Con} (and the references therein) where $a$ in \eqref{1} is gradient-independent.

\begin{theorem}\label{ExLTheorem1} 
Let $\alpha\in (0,1)$ and 
\bqn\label{66a}
a\,,\, f\in C^{2}([0,T]\times[-\ell,\ell]\times\R\times\R) ,\qquad b(\cdot,\pm\ell,\cdot,\cdot)\,,\, g(\cdot,\pm\ell,\cdot,\cdot)\in C^{2}([0,T]\times \R \times \R)
\eqn
satisfy the uniform parabolicity condition  
\begin{equation} \label{upc}
a(t,x,z,p) > 0\,, \qquad \partial_4 b(t,\pm l,z,p) p + b(t,\pm l,z,p) \mp \partial_4 g(t,\pm l,z,p) > 0
\end{equation}
for $(t,x,z,p)\in [0,T]\times [-\ell,\ell]\times \R \times \R$. Let $u^0 \in C^{2+\alpha}([-\ell,\ell])$ satisfy the compatibility condition
\bqn\label{66}
\bsp
\mp b(0,x, u^0,u_x^0){u}_x^0 + g(0,x,u^0,u_x^0) = a(0,x,u^0,{u}_x^0){u}_{xx}^0+f(0,x,u^0,{u}_x^0) \quad \text{ at } \ x=\pm\ell.
\end{split}
\eqn
Then the problem \eqref{1}, \eqref{2} subject to $u(0,\cdot)=u^0$ admits a unique solution $u$ on a maximal interval of existence $[0,\tau_\infty) \subset [0,T]$ such that $u \in \mC^{1+\frac{\alpha}{2},2+\alpha}(\Om_\tau)$ for any $\tau < \tau_\infty$ and  either
\bqn\label{globex}
\lim_{\tau \nearrow \tau_\infty} |u|^{(1+\frac{\alpha}{2},2+\alpha)}_{\Om_\tau} = \infty %
\eqn
or $u \in \mC^{1+\frac{\alpha}{2},2+\alpha}(\Om_T)$.

\noindent Moreover, in the special case \eqref{2a} when $b$ and $g$ are gradient-independent and $b(\cdot,\pm\ell,\cdot)>0$, then either 
\[
\lim_{\tau \nearrow \tau_\infty}  \big( |u|_{\Om_\tau}+|u_x|_{\Om_\tau}+\la u _x \ra^{(\frac{\alpha}{2}, \alpha)}_{\Om_\tau} \big)   = \infty
\]
or $u \in \mC^{1+\frac{\alpha}{2},2+\alpha}(\Om_T)$.

%{\bf (i)} The problem \eqref{1}, \eqref{2} subject to $u(0,\cdot)=u^0$ admits a unique local solution $u \in \mC^{1+\frac{\alpha}{2},2+\alpha}(\Om_\tau)$ for some $\tau\in (0,T)$. 

%{\bf (ii)} This solution can be continued on a maximal interval of existence $[0,\tau_\infty) \subset [0,T]$, that means $u \in \mC^{1+\frac{\alpha}{2},2+\alpha}(\Om_\tau)$ for any $\tau < \tau_\infty$ and either 
%\[
%\lim_{\tau \nearrow \tau_\infty} %|u|^{(1+\frac{\alpha}{2},2+\alpha)}_{\Om_\tau} = \infty %
%\]
%or $u \in \mC^{1+\frac{\alpha}{2},2+\alpha}(\Om_T)$.  

%{\bf (iii)}  Consider the special case \eqref{2a} when $b$ and $g$ are constant in the fourth argument and $b(\cdot,\pm\ell,\cdot)>0$. Then either 
%\[
%\lim_{\tau \nearrow \tau_\infty} %|u|^{(\frac{1+\alpha}{2},1+\alpha)}_{\Om_\tau} = \infty
%\]
%or $u \in \mC^{1+\frac{\alpha}{2},2+\alpha}(\Om_T)$.

\end{theorem}

\begin{proof} To prove existence of a local solution we formulate the problem as a fixed point equation to which we may apply the contraction mapping principle. To do so, we set 
\begin{align*}
h(t,\pm \ell ,z,p) & := b(t,\pm \ell,z,p) p \mp g (t,\pm \ell,z,p)\ ,\quad && (t,z,p)\in [0,T]\times \R\times \R\ ,\\
 h_0(t,\pm \ell) & := \partial_4 h(t,\pm \ell,u^0(\pm \ell),u^0_x(\pm \ell))\ ,&& t\in [0,T]\ ,\\
a_0(t,x) &:= a(t,x,u^0(x),u^0_x(x))\ ,& &(t,x)\in [0,T]\times [-\ell,\ell]\ .  
\end{align*}
Observe that $h_0(t,\pm \ell) > 0$ thanks to (\ref{upc}). Given $\tau \in (0,T]$ define the Banach spaces
\[
\begin{split}
\mathbb{E}_0(\tau) := \big\{ (v, w, z)\,;\, v\in \mC^{\frac{\alpha}{2},\alpha}(\Om_\tau),\, & w(\cdot,\pm\ell)\in C^{\frac{1+\alpha}{2}}([0,\tau]) ,\, z\in C^{2+\alpha}([-\ell,\ell]),\\
& \; \big(v + a_0 z_{xx}\big)(0,\pm \ell) = \big(w \mp h_0 z_{x}\big)(0,\pm \ell) \big\}
\end{split}
\]
and
$$
\mathbb{E}_1(\tau) := \big\{ u \in \mC^{1+\frac{\alpha}{2},2+\alpha}(\Om_\tau); \; u_t(\cdot,\pm \ell) \in C^{\frac{1+\alpha}{2}}([0,\tau]) \big\}\ ,
$$
equipped with the norms
$$
\Vert(v,w,z)\Vert_{\mathbb{E}_0(\tau)}:=|v|^{({\frac{\alpha}{2},\alpha})}_{\Om_\tau} +|w(\cdot,\ell)|^{({\frac{1+\alpha}{2}})}_{[0,\tau]}+|w(\cdot,-\ell)|^{({\frac{1+\alpha}{2}})}_{[0,\tau]} + |z|^{(2+\alpha)}_{[-\ell,\ell]} 
$$
and, respectively, 
$$
\Vert u \Vert_{\mathbb{E}_1(\tau)} := |u|^{({1+\frac{\alpha}{2},2+\alpha})}_{\Om_\tau} + |u_t(\cdot, \ell)|^{(\frac{1+\alpha}{2})}_{[0,\tau]}+|u_t(\cdot,- \ell)|^{(\frac{1+\alpha}{2})}_{[0,\tau]}\ .
$$
Then, introducing the operator $L_0$ by setting
$$
(L_0 u)(t,x) := \big( u_t (t,x) - a_0(t,x) u_{xx}(t,x), u_t(t,\pm\ell) \pm h_0(t,\pm\ell)u_x(t,\pm\ell), u(0,x)\big)\ ,\quad (t,x)\in \overline\Omega_\tau\ ,
$$
for $ u\in \mathbb{E}_1(\tau)$, we derive from \eqref{upc} and Proposition \ref{ExLin1} that $L_0 \in \mathcal{L}\big(\mathbb{E}_1(\tau),\mathbb{E}_0(\tau)\big)$ is a topological isomorphism for any $\tau \in (0,T]$ with
\begin{equation}\label{L0}
\sup_{\tau \in (0,T]} \Vert L_0^{-1} \Vert_{\mathcal{L}(\mathbb{E}_0(\tau),\mathbb{E}_1(\tau))} < \infty.
\end{equation}
In order to solve problem \eqref{1}, \eqref{2}, we shall seek for a fixed point of the mapping $\Phi$ given by
\[
\Phi(u):=L_0^{-1}\big(F(u), H(u), u^0\big)\,, 
\]
where
\begin{align*}
F(u) & :=  (a(\cdot,\cdot,u,u_x) - a_0)  u_ {xx} + f(\cdot,\cdot,u,u_x)\,, \\
H(u) (\cdot,\pm\ell) & :=  \pm h_0(\cdot,\pm\ell)   u_x(\cdot,\pm\ell)  \mp h (\cdot,\pm\ell,u(\cdot,\pm\ell),u_x(\cdot,\pm\ell)),
\end{align*}
for $u\in \mathbb{E}_1(\tau)$. For this we show that $\Phi$ is a contraction on the set
\[
\mathcal{V}(\tau) := \{ u \in \mathbb{E}_1(\tau);\; u(0) = u^0\,,\,\Vert u \Vert_{\mathbb{E}_1(\tau)} \leq M \}
\]
when $\tau$ is chosen sufficiently small, where
\begin{equation}\label{M1}
M > 2 \Vert \Phi(u_0) \Vert_{\mathbb{E}_1(T)} + \Vert u_0 \Vert_{\mathbb{E}_1(T)}\,.
\end{equation}
First observe that $\big(F(u),H(u),u(0,\cdot)\big) \in \mathbb{E}_0(\tau)$ for $u \in \mathcal{V}(\tau)$ due to \eqref{66} and $u(0,\cdot)=u^0$.
We then claim that, given $\varepsilon > 0$, there is $\tau^* \in (0,T]$ such that 
\begin{equation}\label{u0}
\big\Vert \big(F(u),H(u),u^0)  - (F(v),H(v),u^0\big)  \big\Vert_{\mathbb{E}_0(\tau)} \leq \varepsilon \Vert u - v \Vert_{\mathbb{E}_1(\tau)} 
\end{equation}
for all $u,v \in \mathcal{V}(\tau)$ and all $0 < \tau < \tau^*$. Indeed, for $u,v \in \mathcal{V}(\tau)$ we have %at $x=\pm\ell$ 
\[
H(u) - H(v) =\int_0^1  \big[\pm h_0  \mp \partial_4 h \big(\cdot,\cdot, u, \sigma u_x + (1-\sigma)v_x\big)\big] (u_x - v_x)\; \rd\sigma 
\]
and thus 
\[
\begin{split}
\big| \big(H(u) - H(v)\big)(\cdot,\pm\ell) \big|^{(\frac{1+\alpha}{2})}_{[0,\tau]} & \leq  \sup_{\sigma \in (0,1)} \big| \big( h_0 - \partial_4 h (\cdot,\cdot, u, \sigma u_x + (1-\sigma)v_x)\big(\cdot,\pm\ell) \big|^{(\frac{1+\alpha}{2})}_{[0,\tau]}\\
&  \quad\qquad \times \big| (u_x - v_x)(\cdot,\pm\ell) \big|_{[0,\tau]} \\
& \quad + \sup_{\sigma \in (0,1)} \big| \big( h_0 - \partial_4 h (\cdot,\cdot, u, \sigma u_x + (1-\sigma)v_x)\big)(\cdot,\pm\ell) \big|_{[0,\tau]} \\
&  \quad\qquad \times \big| (u_x - v_x)(\cdot,\pm\ell) \big|^{(\frac{1+\alpha}{2})}_{[0,\tau]}.
\end{split}
\]
Note that
\[
\big|u_x(\cdot,\pm\ell) - v_x(\cdot,\pm\ell) \big|_{[0,\tau]}  \leq  \tau^{{\frac{1+\alpha}{2}}} \, \big| u_x(\cdot,\pm\ell) - v_x (\cdot,\pm\ell)\big|^{(\frac{1+\alpha}{2})}_{[0,\tau]}  \leq \tau^{{\frac{1+\alpha}{2}}} \, \Vert u - v \Vert_{\mathbb{E}_1(\tau)}
\]
and, since $h_0 - \partial_4 h (\cdot,\cdot, u, \sigma u_x + (1-\sigma)v_x) = 0$ at $(t,x)=(0,\pm\ell)$, we may estimate %at $x=\pm\ell$
\[
\begin{split}
\big| \big( h_0 - \partial_4  h (\cdot,\cdot, u,& \sigma u_x + (1-\sigma)v_x)\big)(\cdot,\pm\ell)\big|_{[0,\tau]} \\
& \leq  \tau^{{\frac{1+\alpha}{2}}} \,
 \big| \big(h_0 - \partial_4 h (\cdot,\cdot, u, \sigma u_x + (1-\sigma)v_x)\big)(\cdot,\pm\ell) \big|^{(\frac{1+\alpha}{2})}_{[0,\tau]}  \\ 
& \leq  c_1(M)\, \tau^{{\frac{1+\alpha}{2}}} 
\end{split}
\]
with a constant $c_1(M)$ depending on second order derivatives of $h$ (here and in the following we ignore possible dependence of constants on the number $\alpha$). Consequently,
\begin{equation}\label{u1}
\big| \big(H(u) - H(v)\big)(\cdot,\pm\ell) \big|^{(\frac{1+\alpha}{2})}_{[0,\tau]} \le c(M) \, \tau^{{\frac{1+\alpha}{2}}}\,\, \Vert u - v \Vert_{\mathbb{E}_1(\tau)}\ .
\end{equation}
Writing 
\[
\begin{split}
F(u) - F(v)  = \ &  f(\cdot,\cdot,u,u_x) - f(\cdot,\cdot,v,v_x) + \big(a(\cdot,\cdot,u,u_x) - a_0\big) (u_ {xx} - v_ {xx}) \\
 & + \big(a(\cdot,\cdot,u,u_x) - a(\cdot,\cdot,v,v_x)\big) v_ {xx}
\end{split}
\]
and observing that $a(\cdot,\cdot,u,u_x) - a_0$, $a(\cdot,\cdot,u,u_x) - a(\cdot,\cdot,v,v_x)$, and $f(\cdot,\cdot,u,u_x) - f(\cdot,\cdot,v,v_x)$ are of lower order with $a(\cdot,\cdot,u,u_x) - a_0 = 0$ at $t=0$, standard interpolation inequalities for parabolic H\"older spaces (see \cite{Kry} and Proposition~\ref{FR1}) together with similar arguments as above imply that
\[
\begin{split}
\big| (a(\cdot,\cdot,u,u_x) - a_0)& \,(u_ {xx} - v_ {xx}) \big|^{({\frac{\alpha}{2},\alpha})}_{\Om_\tau} \\
 & \leq   \tau^{1/2} \, \sup_{x \in (-l,l)} \big| a(\cdot,x,u,u_x) - a_0(\cdot,x) \big|^{(\frac{1+\alpha}{2})}_{[0,\tau]}\, \Vert u - v \Vert_{\mathbb{E}_1(\tau)}\\
& \qquad +  \; c(l) \, \sup_{t \in (0,\tau)} \Vert a(t,\cdot,u,u_x) - a_0(t,\cdot) \Vert_{C^{1}([-l,l])} \,  \Vert u - v \Vert_{\mathbb{E}_1(\tau)} \\
& \leq  c(M,l) \, \big(\tau^{1/2} + \tau^{\alpha/2}\big)  \, \Vert u - v \Vert_{\mathbb{E}_1(\tau)}
\end{split}
\]
as well as 
\[
\big| \big(a(\cdot,\cdot,u,u_x) - a(\cdot,\cdot,v,v_x)\big) \, v_ {xx} \big|^{({\frac{\alpha}{2},\alpha})}_{\Om_\tau} \leq c(M) \, \big( \tau^{\frac{1}{2+\alpha}} + \tau^{\frac{1 + \alpha}{2}}\big) \, \Vert u - v \Vert_{\mathbb{E}_1(\tau)}
\]
and 
\[
\big| f(\cdot,\cdot,u,u_x) - f(\cdot,\cdot,v,v_x) \big|^{({\frac{\alpha}{2},\alpha})}_{\Om_\tau} \leq c(M) \, \big(\tau^{\frac{1}{2+\alpha}} + \tau^{\frac{1 + \alpha}{2}}\big) \, \Vert u - v \Vert_{\mathbb{E}_1(\tau)}\, , 
\]
with constants depending on derivatives of $a$ and $f$ up to second order.
These estimates combined with~\eqref{u1} yield \eqref{u0}.

Next note that \eqref{L0} and \eqref{u0} imply
\begin{equation}\label{u4}
\Vert \Phi(u) - \Phi(v) \Vert_{\mathbb{E}_1(\tau)} \leq \frac{1}{2} \, \Vert u-v \Vert_{\mathbb{E}_1(\tau)}\ ,\qquad u,v \in \mathcal{V}(\tau)\,,
\end{equation}
for $\tau > 0$ small enough.
Also note that \eqref{u4} and \eqref{M1} entail that 
\[
\Vert \Phi(u) \Vert_{\mathbb{E}_1(\tau)} \leq \Vert \Phi(u_0) \Vert_{\mathbb{E}_1(T)} + \frac{1}{2}\big(\Vert u \Vert_{\mathbb{E}_1(\tau)} + \Vert u_0 \Vert_{\mathbb{E}_1(T)}\big) < \frac{M}{2}+\frac{M}{2}= M
\]
for $u\in\mathcal{V}(\tau)$, that is, $\Phi$ maps $\mathcal{V}(\tau)$ into itself. 
Consequently, the contraction mapping principle yields a unique fixed point $u\in\mathcal{V}(\tau)$ for the mapping~$\Phi$ which solves problem \eqref{1}, \eqref{2} subject to $u(0,\cdot)=u^0$. Clearly, this local solution can be extended to a solution $u$ on a maximal interval of existence $[0,\tau_\infty) \subset [0,T]$ such that $u \in \mC^{1+\frac{\alpha}{2},2+\alpha}(\Om_\tau)$ for any $\tau < \tau_\infty$ and either
\[
\lim_{\tau \nearrow \tau_\infty} |u|^{(1+\frac{\alpha}{2},2+\alpha)}_{\Om_\tau} = \infty %
\]
or $u \in \mC^{1+\frac{\alpha}{2},2+\alpha}(\Om_T)$.

Finally, consider the special case \eqref{2a} when $b$ and $g$ are gradient-independent with $b(\cdot,\pm\ell,\cdot)>0$.
Suppose that there are a constant $R>0$ and a sequence $\tau_i \nearrow \tau_\infty$ such that 
\bqnn
\sup_{i\in\N} \big(|u|_{\Om_{\tau_i}}+|u_x|_{\Om_{\tau_i}}+\la u _x \ra^{(\frac{\alpha}{2}, \alpha)}_{\Om_{\tau_i}} \big)\leq R\,.
\eqnn
This clearly implies that in fact 
\begin{equation} \label{80a}
\sup_{\tau < \tau_\infty} \big( |u|_{\Om_\tau} + |u_x|_{\Om_\tau} + \la u _x \ra^{(\frac{\alpha}{2}, \alpha)}_{\Om_\tau} \big)\leq R\,.
\end{equation}
It then follows from equation (\ref{2a}) that 
\begin{equation} \label{80b}
\sup_{\tau < \tau_\infty} |u_t(\cdot,\pm \ell)|_{[0,\tau]} \leq c(R) \,.
\end{equation}
Using this and the fact that  
$$
|u(t,y) - u(s,y)| \leq  \int_{-\ell}^y |u_x(t,z) - u_x(s,z)| \; \rd z + |u(t,-\ell) - u(s,-\ell)|
$$
for $(t,s,y) \in [0,\tau_\infty)^2 \times [-\ell,\ell]$,
we find 
\begin{equation} \label{80d}
\sup_{\tau < \tau_\infty} \la u \ra^{(\frac{\alpha}{2}, \alpha)}_{\Om_\tau} \leq c(R,\ell)\,  (1 + T^{1-\alpha/2}) \,.
\end{equation}
Now, \eqref{80a}, \eqref{80b}, and \eqref{80d} imply 
\begin{equation} \label{80e}
\max\Big\{|\wt{a}|^{(\frac{\alpha}{2}, \alpha)}_{\Om_\tau}, \; |\wt{f}|^{(\frac{\alpha}{2}, \alpha)}_{\Om_\tau}, \; |\wt{b}(\cdot,\pm\ell)|^{(\frac{1+\alpha}{2})}_{[0,\tau]},\; |\wt{g}(\cdot,\pm\ell)|^{(\frac{1+\alpha}{2})}_{[0,\tau]}\Big\} \leq c(R,\ell,T)
\end{equation}
for each $\tau < \tau_\infty$, where $\wt{a}(t,x) := a(t,x,u(t,x), u_x(t,x))$ and $\wt{f}$, $\wt{b}$, $\wt{g}$ are defined analogously. 

As $a$ and $b$ are strictly positive on $[0,T] \times [-\ell,\ell] \times [-R, R]^2$, respectively, on $[0,T] \times \{\pm\ell\} \times [-R, R]$, Proposition~\ref{ExLin1} ensures the existence of a constant $c=c\big(R, \ell, T, |u^0|^{(2+\alpha)}_\Om\big)$ such that
\bqnn\label{81}
|u|_{\Om_\tau}^{(1+\frac{\alpha}{2},2+\alpha)}\leq c\,, \quad \tau < \tau_\infty\,,
\eqnn
hence $u \in \mC^{1+\frac{\alpha}{2},2+\alpha}(\Om_T)$ by our previous findings. This proves the theorem.
\end{proof}

Actually, a closer look at the proof shows that the global existence criterion~\eqref{globex} for the general (i.e. gradient-dependent) case can be weakened. Indeed, it suffices to control
$$
|u|_{\Omega_\tau}^{(\frac{1+\alpha}{2},\alpha)} + \big|u(\cdot,\pm\ell)\big|^{(\frac{1+\alpha}{2})}_{[0,\tau]}
$$
uniformly in $\tau<\tau_\infty$ in this case. Moreover, since the maximal regularity result stated in Theorem~\ref{ExLin1} holds also true in higher space dimensions (see \cite[Theorem 1.1]{Biz}), one easily verifies that Theorem~\ref{ExLTheorem1} is true in this case, too. 

%The a priori estimate \eqref{80} required for global existence can be derived in certain cases for the more general case of \eqref{2} as shown in the following section.

%%%%%%%%%%%%%%%%%%%%%%%%%%%%%%%
%%%%%%%%%%%%%%%%%%%%%%%%%%%%%%%
\section{A Priori Estimates and Global Solutions} \label{sec:globsol}
%%%%%%%%%%%%%%%%%%%%%%%%%%%%%%%
%%%%%%%%%%%%%%%%%%%%%%%%%%%%%%%

Theorem~\ref{ExLTheorem1} reduces the question of global existence  to \eqref{1} with gradient-independent boundary conditions \eqref{2a} to finding {\it a priori} estimates for 
\bqn\label{uo}
|u|_{\Om_\tau}+|u_x|_{\Om_\tau}+\la u _x \ra^{(\frac{\alpha}{2}, \alpha)}_{\Om_\tau} 
\eqn
independent of $\tau<\tau_\infty$ for the solution $u$ 
on the maximal interval of existence $[0,\tau_\infty) \subset [0,T]$. The aim of this section is to further reduce this condition. More precisely, we shall show $L_\infty$- and H\"older-bounds on the gradient $u_x$ solely based on bounds on $|u|_{\Om_\tau}=\sup_{(t,x)\in\Om_\tau} |u(t,x)|$, that is, we show that a bound on the first term of \eqref{uo} implies bounds on the second and third terms. It is worthwhile to point out that we can obtain such estimates  even for equation \eqref{1} combined with the gradient-dependent boundary condition~\eqref{2}. While H\"older-estimates on the gradient, i.e. estimates on $\la u _x \ra^{(\frac{\alpha}{2}, \alpha)}_{\Om_\tau}$, are rather easy to obtain from the existing theory (see Subsection~\ref{Sec1}), more effort has to be invested in Subsection~\ref{Abschnitt_4} to derive estimates on $|u_x|_{\Om_\tau}$.

\subsection{{\it A priori} estimates on the gradient}\label{Abschnitt_4}

We shall find {\it a priori} $L_\infty$-bounds on the gradient $u_x$ of a solution to \eqref{1}, \eqref{2} presupposing a bound on its $L_\infty$-norm. The proof is in the spirit of \cite{Ter} and uses Kruzhkov's  idea of introducing a new variable \cite{Kru1, Kru2}.  In the next subsection we derive gradient estimates when imposing the Bernstein-Nagumo condition. In  Subsection \ref{subsection_2} we then indicate how to weaken this condition for a right hand side $f+f_1$ in \eqref{1} where only $f$ satisfies the Bernstein-Nagumo condition and $f_1$ is allowed to be unbounded in the gradient variable. We also consider the case of mixed boundary conditions in Subsection~\ref{subsection_3}.\\

In the following, a {\it classical solution} to \eqref{1}, \eqref{2} [resp. to \eqref{1}, \eqref{2a}] on $\Om_\tau$ with $\tau\le T$ is a function $u\in C^{1,2}(\Om_\tau)\cap C(\overline{\Omega}_\tau)$ with derivatives $u_t(\cdot,\pm\ell)$ and $u_x(\cdot,\pm\ell)$ being defined on $(0,\tau)$ and satisfying  \eqref{1}, \eqref{2} [resp. \eqref{1}, \eqref{2a}] pointwise in~$\Om_\tau$. Note that the existence and uniqueness of such a solution (with higher regularity) is guaranteed by Theorem~\ref{ExLTheorem1} when imposing the assumptions stated there.

Throughout we assume continuity of the data, that is,
\bqnn
a, f\in C([0,T]\times[-\ell,\ell]\times\R\times\R) ,\qquad b(\cdot,\pm\ell,\cdot,\cdot), g(\cdot,\pm\ell,\cdot,\cdot)\in C([0,T]\times \R\times \R)
\eqnn
and the parabolicity condition $a>0$ and $b(\cdot,\pm\ell,\cdot,\cdot)>0$. Stronger assumptions will be indicated explicitly.

\subsubsection{\textbf{Gradient estimates under the Berstein-Nagumo condition}}\label{subsection_1}

The next theorem provides {\it a priori} estimates on the gradient for any  classical solution $u$ to \eqref{1}, \eqref{2} on $\Om_\tau$ with $\tau\le T$ in dependence on its $L_\infty$-bound.

\begin{theorem}\label{A1}
Let $u^0$ be  Lipschitz continuous on $[-\ell,\ell]$, that is,
\begin{equation}\label{10}
|u^0(x)-u^0(y)|\le K|x-y|, \quad x,y\in [-\ell,\ell] .
\end{equation}
Let $M>0$ and suppose that there are $q_0\geq K$ and $\psi\in C^1\big([0,\infty), [1,\infty)\big)$ with 
\begin{equation}\label{9}
\int\limits_{q_0}^\infty\frac{\rho\,\rd\rho}{\psi(\rho)}> 2M,
\end{equation}
such that
\begin{equation}\label{6}
|f(t,x,z,p)|\le a(t,x,z,p)\psi(|p|), \quad (t,x,z,p)\in\overline{\Omega}_T\times[-M,M]\times\R
\end{equation}
and
\begin{align}
\pm g(t, \ell, z, \pm p)    & \le b(t,  \ell, z, \pm p)p ,\qquad  %\label{9aNEU}\\
\mp g(t, - \ell, z, \pm p)  \le b(t, - \ell, z, \pm p)p , \label{9bNEU}
%g(t, l, z, p)    & \le b(t,  l, z, p)p , \label{9a}\\
%-g(t, - l, z, p) & \le b(t, - l, z, p)p , \label{9b}\\
%g(t,- l, z,- p)  & \le b(t, - l, z, - p)p  , \label{9c}\\
%-g(t,  l, z, -p) & \le b(t,  l, z, -p)p , \label{9d}
\end{align}
for $(t,z)\in (0,T]\times [-M,M]$ and $p\ge q_0$. 
Then there exists a constant
$M_1:=M_1(M,\psi,K)$ such that if $u$ is any  classical solution to \eqref{1}, \eqref{2} on $\Om_\tau$ with $\tau\le T$ subject to $u(0,\cdot)=u^0$ satisfying $|u|_{\Om_\tau}\le M$, then  $|u_x|_{\Om_\tau}\leq M_1$.
\end{theorem}

\begin{proof} We adapt the proof of \cite{Ter} to the case of the dynamic boundary condition \eqref{2}. Due to \eqref{9} there is $q_1>q_0$ with
\begin{equation}\label{265}
\int\limits_{q_0}^{q_1}\frac{\rho\,\rd\rho}{\psi(\rho)}=2M.
\end{equation}
Let $ \kappa$ be defined by
\bqnn
\kappa(\xi):=\int_{\xi}^{q_1} \frac{\rd \rho}{\psi (\rho)}, \quad \xi\in [q_0,q_1].
\eqnn
Since $\psi$ is positive,  $\kappa$ is strictly decreasing on $[q_0,q_1]$ with $\kappa(q_1)=0$.
Thus, its inverse $q:=\kappa^{-1}$ is decreasing on the interval $[0,\kappa_0]$, where $\kappa_0:=\kappa(q_0)$, with $q(0)=q_1$ and $q(\kappa_0)=q_0$.
Define
\bqnn
h(\xi):=\int\limits_{q(\xi)}^{q_1}\frac{\rho\,\rd\rho}{\psi(\rho)}\,,\quad \xi\in [0,\kappa_0]\, .
\eqnn
Noticing that
\bqn\label{pp}
\begin{split}
            h'(\xi) =q(\xi)\geq q_0\,,\qquad
						h''(\xi)=-\psi(q(\xi))\,,
\end{split}
\eqn
we see that the function $h$ solves 
\begin{equation}\label{11}
\begin{split}
h''(\xi)+\psi(|h'(\xi)|)&=0,\quad \,\,\xi\in (0,\kappa_0),\\
%h'(\xi)&\geq q_0,\quad \xi\in (0,\kappa_0),\\
h(\kappa_0)=2M,&\quad 
h(0)   =0.
\end{split}
\end{equation}
Since $h'\geq q_0\geq K$ and $h(0)=0$, there is for $x,y\in [-\ell,\ell]$ with $|x-y|\leq \kappa_0$ some $\xi\in [0,|x-y|]$ such that
$$
h(|x-y|)=h'(\xi)|x-y|\geq K|x-y|.
$$
Hence, by \eqref{10}, we have
\begin{equation}\label{275}
|u^0(x)-u^0(y)|\leq h(|x-y|), \quad x,y\in [-\ell,\ell], \;|x-y|\leq \kappa_0.
\end{equation}
Let us define the sets
\begin{equation}\label{13}
\begin{split}
S_1& :=\{(t,x,y)\in\overline{P}:\,-\ell< x < -\ell+\kappa_0, \; y=-\ell, \; t\in (0,\tau]\},\\
S_2& :=\{(t,x,y)\in\overline{P}:\,x=\ell, \; \ell-\kappa_0 < y<\ell, \; t\in (0,\tau]\},\\
S_3& :=\{(t,x,y)\in\overline{P}:\,x=y, \; t\in (0,\tau]\},\\
S_4& :=\{(t,x,y)\in\overline{P}:\,x-y=\kappa_0, \; t\in (0,\tau]\},
\end{split}
\end{equation}
where
$$
P  :=\big\{(x,y)\in (-\ell,\ell)^2:\,0<x-y<\kappa_0 \big\},\qquad
P_\tau  :=(0,\tau)\times P.
$$
%which, for fixed $t$, are illustrated in the $xy$-plane by the picture
%\begin{center}
%\includegraphics[width=10cm]{bild2.eps}
%\end{center}
We put
$$
S  :=\bigcup_{i=1}^4 S_i,\qquad
B  :=\{0\}\times \overline{P},\qquad
\Gamma  :=S\cup B
$$
and define the auxiliary functions $v$, $w$, and  $\wt{w}$ by
\begin{equation}\label{14}
\begin{split}
v(t,x,y)&:=u(t,x)-u(t,y),\\
w(t,x,y)&:= v(t,x,y)-h(x-y),\\
\wt{w}(t,x,y)&:=e^{-t}w(t,x,y)
\end{split}
\end{equation}
for $(t,x,y)\in\overline{P}_\tau$.
Derivatives with respect to the second and third variable we denote as derivatives with respect to $x$ and $y$, respectively. To keep notation as simple as possible we also use 
$u_x :=  u_y := \partial_2 u$. We shall show that $\wt{w}$ does not attain a positive maximum in $\overline{P}_\tau$. Let $(t,x,y)\in P_\tau\setminus \Gamma $ and note that $(t,x)$ and $(t,y)$ belong to $\Omega_\tau$. Thus 
\bqnn%\label{273}
-u_t(t,x)+a(t,x,u,u_x)u_{xx}+f(t,x,u,u_x)=
-u_t(t,y)+a(t,y,u,u_y)u_{yy}+f(t,y,u,u_y)=0 .                        %\label{17}
\eqnn
by \eqref{1}, where we use here and in the following the notation
$$
a(t,y,u,u_y)u_{yy}:=a(t,y,u(t,y),u_y(t,y))u_{yy}(t,y)\,.
$$
Subtracting the two equations and using the definition of $v$ we obtain
\begin{equation}\label{259}
\bsp
-v_t(t,x,y)  +a(t,x,u,v_x)v_{xx}+f(t,x,u,v_x)
             +a(t,y,u,-v_y)v_{xx}-f(t,y,u,-v_y)=0.
\end{split}
\end{equation}
Recalling from \eqref{6} that
\begin{equation}\label{260}
\begin{split}
f(t,x,u,v_x)  \le a(t,x,u,v_x)\psi(|v_x|),\qquad
-f(t,y,u,-v_y) & \le a(t,y,u,-v_y)\psi(|v_y|)
\end{split}
\end{equation}
and using the notation
\bqnn%\label{18}
A(t,x):=a(t,x,u,v_x),\qquad
A(t,y):=a(t,y,u,-v_y).
\eqnn
we derive from \eqref{259} and \eqref{260} the inequality
\begin{equation*}
-v_t(t,x,y)+A(t,x)(v_{xx}+\psi(|v_x|))+A(t,y)(v_{yy}+\psi(|v_y|)) \ge 0.
\end{equation*}
Note that \eqref{11} implies
\begin{equation*}
A(t,x)(h''(x-y)+\psi(|h'(x-y)|))+A(t,y)(h''(x-y)+\psi(|h'(x-y)|)) = 0.
\end{equation*}
Subtracting the previous (in-)equalities yields
\bqnn%\label{274}
-w_t(t,x,y)+A(t,x)w_{xx}(t,x,y)+A(t,y)w_{yy}(t,x,y)+r(t,x,y) \ge 0,
\eqnn
where
\begin{equation*}%\label{26}
r(t,x,y):=\psi(|v_x(t,x,y)|)-\psi(|h'(x-y)|)+\psi(|v_y(t,x,y)|)-\psi(|h'(x-y)|),
\end{equation*}
that is,
\begin{equation}\label{27}
-\wt{w}_t-\wt{w}+A(t,x)\wt{w}_{xx}+A(t,y)\wt{w}_{yy}+re^{-t}\ge 0, \quad (t,x,y)\in\overline{P}_\tau\setminus \Gamma .
\end{equation}
Assume for contradiction that the function $\wt{w}$ attains its positive maximum at $(t_0,x_0,y_0)\in\overline{P}_\tau\setminus \Gamma$. At this point of maximum there holds
\begin{equation}\label{29}
\begin{split}
\wt{w}(t_0,x_0,y_0)  > 0, \quad \wt{w}_t(t_0,x_0,y_0) \ge 0, \quad
\wt{w}_{xx}(t_0,x_0,y_0)  \le 0, \quad \wt{w}_{yy}(t_0,x_0,y_0) \le 0.\\
\end{split}
\end{equation}
Moreover,
\begin{equation*}
\wt{w}_{x}(t_0,x_0,y_0) = \wt{w}_{y}(t_0,x_0,y_0) = 0
\end{equation*}
from which, by definition of $\wt{w}$, we get
\begin{equation*}
e^{-t}(v_x(t_0,x_0,y_0)-h'(x_0-y_0))=e^{-t}(v_y(t_0,x_0,y_0)+h'(x_0-y_0))=0.
\end{equation*}
Thus
\begin{equation}\label{228}
v_x(t_0,x_0,y_0)=-v_y(t_0,x_0,y_0)=h'(x_0-y_0)
\quad\text{ and }\quad
r(t_0,x_0,y_0)=0.
\end{equation}
Since both $A(t_0, x_0)$ and $A(t_0, y_0)$ are non-negative  it follows from \eqref{29} and \eqref{228} that
\begin{equation}\label{230}
-\wt{w}_t-\wt{w}+A(t,x)\wt{w}_{xx}+A(t,y)\wt{w}_{yy}+re^{-t} < 0 \quad \text{at }\ (t,x,y)= (t_0,x_0,y_0)
\end{equation}
which contradicts \eqref{27}. Therefore, $\wt{w}$ does not attain a positive maximum in $\overline{P}_T\setminus \Gamma$.

We next show that $\wt{w}$ does also not attain a positive maximum on $\Gamma$ for which we distinguish the two cases $\kappa_0<2\ell$ and $\kappa_0\geq 2\ell$. 

{\bf Case 1: $\kappa_0 < 2\ell$}. In this case we have $-\ell<y<\ell$ for $(t,\ell,y)\in S_2$. Thus, the function $u$ satisfies equations~\eqref{1} and \eqref{2} at the points $(t,y)$ and $(t,\ell)$, that is,
\begin{equation*}
\begin{split}
-u_t(t,y)+a(t,y,u,u_y)u_{yy}+f(t,y,u,u_y) & =0,\\
-u_t(t,\ell)-b(t,\ell,u,u_x)u_{x}+g(t,\ell,u,u_x) & =0.
\end{split}
\end{equation*}
Subtracting the two equations and recalling \eqref{6} and the definitions of $v$ and $A(t,y)$, we derive
the inequality 
\begin{equation}\label{37}
-v_t(t,\ell,y)+A(t,y)(v_{yy}+\psi(|v_y|))+r_1(t,\ell,y) \ge 0,
\end{equation}
where
\begin{equation*}
r_1(t,\ell,y):=-b(t,\ell,u,v_x)v_{x}+g(t,\ell,u,v_x) \,.
\end{equation*}
Since $(t,\ell,y)\in S_2P_T$ and hence $0 \leq \ell-y \leq \kappa_0$, it follows from \eqref{11} that
\begin{equation*}
A(t,y)(h''(\ell-y)+\psi(|h'(\ell-y)|)) = 0 .
\end{equation*}
Subtracting this from \eqref{37} and using the definition of $\wt{w}$, we derive
\begin{equation}\label{41}
-\wt{w}_t-\wt{w}+A(t,y)\wt{w}_{yy}+r_2e^{-t}+r_1e^{-t}\ge 0,\quad (t,\ell,y)\in S_2,
\end{equation}
where
\begin{equation*}
r_2(t,\ell,y):=A(t,y)(\psi(|v_y(t,\ell,y)|)-\psi(|h'(\ell-y)|)).
\end{equation*}
Assume now for contradiction that
$\wt{w}$ attains its maximum at some point $(t_0,\ell,y_0)\in S_2$. At this maximum point there holds
\begin{equation*}
\begin{split}
&\wt{w}(t_0,\ell,y_0) > 0, \qquad \wt{w}_t(t_0,\ell,y_0) \ge 0, \qquad \wt{w}_{yy}(t_0,\ell,y_0) \le 0 .
\end{split}
\end{equation*}
Furthermore, $(t_0,\ell,y_0)$ is an inner point of $S_2$, hence
$\wt{w}_{y}(t_0,\ell,y_0) = 0$.
Since, by definition of $\wt{w}$,
\begin{equation*}
e^{-t_0}(v_y(t_0,\ell,y_0)+h'(\ell-y_0))=0
\end{equation*}
we have $r_2(t_0,\ell,y_0)=0$.
Also note that
\begin{equation*}
\wt{w}_{x}(t_0,\ell,y_0)=e^{-t_0}(v_x(t_0,\ell,y_0)-h'(\ell-y_0))\ge 0,
\end{equation*}
that is, using \eqref{pp},
\begin{equation}\label{233}
v_x(t_0,\ell,y_0)  \ge h'(\ell-y_0) \ge q_0>0
\end{equation}
and consequently, due to \eqref{9bNEU}, 
\begin{equation*}
r_1(t_0,\ell,y_0)\overset{}{=} -b(t_0,\ell,u,v_x)v_{x}+g(t_0,\ell,u,v_x) \le 0.
\end{equation*}
In summary, we have
\begin{equation}\label{234}
-\wt{w}_t-\wt{w}+A(t,y)\wt{w}_{yy}+r_2e^{-t_0}+r_1e^{-t_0}< 0 \quad \text{at }\ (t,x,y)=(t_0,\ell,y_0),
\end{equation}
in contradiction to \eqref{41}. Therefore, $\wt{w}$ does not attain a positive maximum in $S_2$.
In the same way we consider $\wt{w}$ on $S_1$. Given $(t,x,-\ell)\in S_1$, we derive from
\begin{equation*}
\begin{split}
-u_t(t,x)+a(t,x,u,u_x)u_{xx}+f(t,x,u,u_x) =     %\label{47}
-u_t(t,-\ell)+b(t,-\ell,u,u_x)u_{x}+g(t,-\ell,u,u_x)  =0     %\label{48}
\end{split}
\end{equation*}
by subtraction and by using \eqref{6} the inequality
\begin{equation*}                                                   %\label{51}
-v_t(t,x,-\ell)+A(t,x)(v_{xx}+\psi(|v_x|))+r_3(t,x,-\ell) \ge 0
\end{equation*}
where
\begin{equation*}                                                    %\label{51}
r_3(t,x,-\ell):=b(t,-\ell,u,-v_y)v_{y}-g(t,-\ell,u,-v_y).
\end{equation*}
Since, due to \eqref{11},
\begin{equation*}                                                    %\label{52}
A(t,x)(h''(x+\ell)+\psi(|h'(x+\ell)|)) = 0
\end{equation*}
we further obtain
\begin{equation}\label{56}
-\wt{w}_t-\wt{w}+A(t,x)\wt{w}_{xx}+r_4e^{-t}+r_3e^{-t}\ge 0, \quad (t,x,-\ell)\in S_1
\end{equation}
with
\begin{equation*}                                                   %\label{57}
r_4(t,x,-\ell):=A(t,x)(\psi(|v_x(t,x,-\ell)|)-\psi(|h'(x+\ell)|)).
\end{equation*}
Assume for contradiction that  $\wt{w}$ attains its positive maximum at a point $(t_0,x_0,-\ell)\in S_1$ with $x_0\neq \ell$. 
%Then
%\begin{equation*}                                                   %\label{58}
%\begin{split}
%\wt{w}(t_0,x_0,-l) & > 0, \quad \wt{w}_t(t_0,x_0,-l) \ge 0, \\
%-\wt{w}_{y}(t_0,x_0,-l) & \ge 0,\quad \wt{w}_{x}(t_0,x_0,-l) = 0, \\
%\wt{w}_{xx}(t_0,x_0,-l) & \le 0.\\
%\end{split}
%\end{equation*}
From $\wt{w}_{x}(t_0,x_0,-\ell)=0$ we deduce $r_4(t_0,\ell,y_0)=0$.
Furthermore, since
\begin{equation*}                                                    %\label{60}
0\le -\wt{w}_{y}(t_0,x_0,-\ell)=e^{-t_0}(-v_y(t_0,x_0,-\ell)-h'(x_0+\ell)),
\end{equation*}
it follows from \eqref{pp} that
\begin{equation*}
-v_y(t_0,x_0,-\ell)  \ge h'(x_0+\ell) \ge q_0>0,
\end{equation*}
and thus $r_3(t_0,x_0,-\ell)\le 0$ due to  \eqref{9bNEU}. Consequently,
\begin{equation*}
-\wt{w}_t-\wt{w}+A(t,x)\wt{w}_{xx}+r_4e^{-t}+r_3e^{-t}< 0 \quad \text{ at }\ (t,x,y)=(t_0,x_0,-\ell)
\end{equation*}
in contradiction to\eqref{56}. Thus, $\wt{w}$ does not attain a positive maximum on $S_1$.

Next, we consider $\wt{w}$ on $S_3$ and $S_4$. Since $x=y$ for $(t,x,y)\in S_3$ and $h(0)=0$, we clearly have $\wt{w}=0$ on $S_3$ so that $\wt{w}$ does not attain a positive maximum on $S_3$.
Given $(t,x,y)\in S_4$ we have $x-y=\kappa_0$. Now \eqref{11} implies $h(\kappa_0)=2M = 2|{u(t,x)}|_{\Om_\tau}$, hence
\begin{equation*}
\wt{w}(t,x,y) = u(t,x)-u(t,y)-h(\kappa_0)\le 0,
\end{equation*}
and $\wt{w}$ thus does not attain a positive maximum on $S_4$ either.

Finally, since for $(t,x,y)\in B$ we have $t=0$, \eqref{275} yields
\begin{equation*}
\wt{w}(t,x,y) = (u^0(x)-u^0(y)-h(x-y))\le 0,
\end{equation*}
and $\wt{w}$ does not attain a positive maximum on $\Gamma$ if $\kappa_0 < 2\ell$.

{\bf Case 2: $\kappa_0 \geq 2\ell$}. In this case, the set $S_4$ is empty while the sets $S_1$ and $S_2$ also contain the set $\{(t,l,-\ell): t\in (0,\tau]\}$. We have to show therefore that $\wt{w}$ does not attain a positive maximum on this set.
For, note that at the points $(t,\ell)$ and $(t,-\ell)$ the function $u$ satisfies according to \eqref{2}
\begin{equation*}
\begin{split}
-u_t(t,\ell)-b(t,\ell,u,u_x)u_{x}+g(t,\ell,u,u_x) =
-u_t(t,-\ell)+b(t,-\ell,u,u_x)u_{x}+g(t,-\ell,u,u_x)  =0
\end{split}
\end{equation*}
from which we get
\begin{equation}\label{262}
\bsp
-w_t(t,\ell,-\ell)  -b(t,\ell,u,u_x)u_{x}+g(t,\ell,u,u_x) -b(t,-\ell,u,u_x)u_{x}-g(t,-\ell,u,u_x)=0.
\end{split}
\end{equation}
Assume for contradiction that $\wt{w}$ attains a positive maximum at some point $(t_0,\ell,-\ell)$ with $t_0\in (0,\tau]$. Then
\begin{equation*}
\begin{split}
 \wt{w}_{x}(t_0,\ell,-\ell) & =e^{-t}(u_x(t_0,\ell)-h'(2\ell))\ge 0, \\
-\wt{w}_{y}(t_0,\ell,-\ell) & =e^{-t}(u_y(t_0,-\ell)-h'(2\ell))\ge 0
\end{split}
\end{equation*}
and \eqref{11} implies
\begin{equation*}
\begin{split}
u_x(t_0,\ell) & \ge h'(2\ell)\ge q_0 > 0, \\
u_y(t_0,-\ell) & \ge h'(2\ell)\ge q_0 > 0.
\end{split}
\end{equation*}
From this and \eqref{9bNEU} it follows that
\begin{equation*}
\begin{split}
-b(t,\ell,u,u_x)u_{x}+g(t,\ell,u,u_x)  \le 0 , \qquad
-b(t,-\ell,u,u_x)u_{x}-g(t,-\ell,u,u_x)  \le 0
\end{split}
\end{equation*}
and hence, from \eqref{262}, that
\bqn\label{263}
-w_t(t_0,\ell,-\ell)\ge 0.
\eqn
On the other hand, since $\wt{w}$ attains on $\{(t,\ell,-\ell):\; t\in (0,\tau]\}$ a positive maximum, we have
$$
\wt{w}(t_0,\ell,-\ell)>0, \quad \wt{w}_t(t_0,\ell,-\ell)\ge 0
$$
and further
$$
-w_t(t_0,\ell,-\ell)=-e^t(\wt{w}(t_0,\ell,-\ell)+\wt{w}_t(t_0,\ell,-\ell))< 0
$$
contradicting \eqref{263}. So, $\wt{w}$ does also not attain a positive maximum on $\Gamma$ if $\kappa_0 \ge 2\ell$.
\\

Consequently, we have shown that $\wt{w}$ does not attain a positive maximum in $\overline{P}_\tau$ and thus
\bqnn
\wt{w}(t,x,y) \le 0, \quad (t,x,y)\in\overline{P}_\tau.
\eqnn
In the same way one shows that the function $\wt{w}_1$, defined by
\bqnn
\wt{w}_1(t,x,y):=u(t,y)-u(t,x)-h(x-y),\quad  (t,x,y)\in  \overline{P}_\tau,
\eqnn
satisfies
\bqnn
\wt{w}_1(t,x,y) \le 0, \quad (t,x,y)\in\overline{P}_\tau.
\eqnn
Together we deduce
\bqnn
|u(t,x)-u(t,y)|\le h(x-y), \quad (t,x,y)\in\overline{P}_\tau\,,
\eqnn
and putting
\bqnn
Q_\tau:=\{(t,x,y): (t,y,x)\in P_\tau\}\,,
\eqnn
and using symmetry with respect to the variables we conclude
\begin{equation*}
|u(t,x)-u(t,y)|\le h(|x-y|), \quad (t,x,y)\in\overline{Q}_\tau\cup \overline{P}_\tau,
\end{equation*}
with
\begin{equation*}
\overline{Q}_\tau \cup \overline{P}_\tau=\{(t,x,y)\in\R^3: \; 0\le|x-y|\le\kappa_0, \; |x|\le \ell, \; |y|\le \ell, \; t\in [0,\tau]\}.
\end{equation*}
Since $h(0)=0$ we obtain
\begin{equation}\label{264}
|u_x(t,x)|\le h'(0)=q_1, \quad (t,x)\in \Om_\tau.
\end{equation}
According to \eqref{265} the number $q_1$ depends only on $K$, $M$, $q_0$, and $\psi$. This proves Theorem~\ref{A1}
\end{proof}
%%%%%%%%%%%%%%%%%%%%%%%%%%%%%%%%%%%%%%%%%%%%%%%%%%%%%%%%%%%%%%%%%%%%%%%%%%%%%%%%%%%%%%%%%%%%%%%%%%%%%%%%%%%%%%%%%%%%%%%%%%%%%%%%%%%%%%%%%%%%%%%%%%%%%%%%%%%%%%%%%%%%%%%%%%%%%%%%%%%%%%%%%%%%%%%%%%
%%%%%%%%%%%%%%%%%%%%%%%%%%%%%%%%%%%%%%%%%%%%%%%%%%%%%%%%%%%%%%%%%%%%%%%%%%%%%%%%%%%%%%%%%%%%%%%%%%%%%%%%%%%%%%%%%%%%%%%%%%%%%%%%%%%%%%%%%%%%%%%%%%%%%%%%%%%%%%%%%%%%%%%%%%%%%%%%%%%%%%%%%%%%%%%%%%

%\begin{remark}\label{Bemerkung6}
It may be worthwhile to note that \eqref{265} yields more information on the gradient bound $M_1=q_1$. 
Condition \eqref{9bNEU} is needed because of the dynamical boundary
condition \eqref{2}. A simple situation for which \eqref{9bNEU} holds is obtained by strengthening condition \eqref{9} to
\bqn\label{266}
\int\limits_{}^\infty\frac{\rho\,\rd\rho}{\psi(\rho)} = \infty .
\eqn
Indeed, in this case there exists for each $q_0>0$ some $q_1>0$ such that \eqref{265} is satisfied. Then \eqref{9bNEU} holds e.g. if $b\ge \delta>0$ and $g$ is bounded.
%\end{remark}
%%%%%%%%%%%%%%%%%%%%%%%%%%%%%%%%%%%%%%%%%%%%%%%%%%%%%%%%%%%%%%%%%%%%%%%%%%%%%%%%%%%%%%%%%%%%%%%%%%%%%%%%%%%%%%%%%%%%%%%%%%%%%%%%%%%%%%%%%%%%%%%%%%%%%%%%%%%%%%%%%%%%%%%%%%%%%%%%%%%%%%%%%%%%%%%%%%
%%%%%%%%%%%%%%%%%%%%%%%%%%%%%%%%%%%%%%%%%%%%%%%%%%%%%%%%%%%%%%%%%%%%%%%%%%%%%%%%%%%%%%%%%%%%%%%%%%%%%%%%%%%%%%%%%%%%%%%%%%%%%%%%%%%%%%%%%%%%%%%%%%%%%%%%%%%%%%%%%%%%%%%%%%%%%%%%%%%%%%%%%%%%%%%%%%

\subsubsection{\textbf{Gradient estimates under a weakened Bernstein-Nagumo condition}}\label{subsection_2}

In \cite{Ter} it was shown for classical (i.e. Dirichlet, Neuman, or Robin) boundary conditions  that gradient estimates hold true for the more general situation that a term $f_1(t,x,u,u_x)$ is added on the right hand side of \eqref{1} which may arbitrarily increase in the  gradient variable provided it satisfies a homogeneity condition (see \eqref{225} below). We similarly extend Theorem~\ref{A1} for dynamical boundary conditions.

More precisely, we may consider 
\begin{equation}\label{223}
u_t-a(t,x,u,u_x)u_{xx}=f(t,x,u,u_x)+f_1(t,x,u,u_x), \quad (t,x)\in(0,T)\times(-\ell,\ell),
\end{equation}
together with the boundary condition
\begin{equation}\label{224}
u_t\pm b(t,x,u,u_x)u_{x}=g(t,x,u,u_x)+g_1(t,x,u,u_x), \quad t\in(0,T), \quad x=\pm\ell\,,
\end{equation}
where we impose on $f_1$ and $g_1$ the following conditions:
\bqn\label{225}
\bsp
f_1(t,y,z_1,\pm p) -f_1(t,x,z_2,\pm p) \geq 0 
\end{split}
\eqn
for $-\ell \leq x \leq y \leq \ell$, $-M\leq z_1 \leq z_2 \leq M$, $p\geq 0$,
\bqn\label{226}
\bsp
f_1(t,x,z_1,\pm p_1) -g_1(t,\ell,z_2,\pm p_2) \geq 0
\end{split}
\eqn
for $-\ell \leq x \leq \ell$,\; $-M\leq z_1 \leq z_2 \leq M$,\; $q_0 \leq p_1 \leq p_2$, and
\bqn\label{227}
\bsp
g_1(t,\ell,z_1,-p_1) & -f_1(t,x,z_2,-p_2) \geq 0, \qquad g_1(t,-\ell,z_1,p_1)  -f_1(t,x,z_2,p_2) \geq 0
\end{split}
\eqn
for $-\ell \leq x \leq \ell$,\; $-M\leq z_1 \leq z_2 \leq M$,\; $q_0 \leq p_2 \leq p_1$.
 Then we can prove:

%%%%%%%%%%%%%%%%%%%%%%%%%%%%%%%%%%%%%%%%%%%%%%%%%%%%%%%%%%%%%%%%%%%%%%%%%%%%%%%%%%%%%%%%%%%%%%%%%%%%%%%%%%%%%%%%%%%%%%%%%%%%%%%%%%%%%%%%%%%%%%%%%%%%%%%%%%%%%%%%%%%%%%%%%%%%%%%%%%%%%%%%%%%%%%%%%%
\begin{theorem}\label{Theorem_11}
Let $M>0$ and suppose \eqref{10}-\eqref{9bNEU} and \eqref{225}--\eqref{227}.
Then there exists a constant
$M_1:=M_1(q_0, M,\psi,K)$ such that if $u$ is any  classical solution to \eqref{223}, \eqref{224} on $\Om_\tau$ with $\tau\le T$ subject to $u(0,\cdot)=u^0$ satisfying $|u|_{\Om_\tau}\le M$, then  $|u_x|_{\Om_\tau}\leq M_1$.
\end{theorem}
%%%%%%%%%%%%%%%%%%%%%%%%%%%%%%%%%%%%%%%%%%%%%%%%%%%%%%%%%%%%%%%%%%%%%%%%%%%%%%%%%%%%%%%%%%%%%%%%%%%%%%%%%%%%%%%%%%%%%%%%%%%%%%%%%%%%%%%%%%%%%%%%%%%%%%%%%%%%%%%%%%%%%%%%%%%%%%%%%%%%%%%%%%%%%%%%%%
\begin{proof}
Except for small changes the proof is the same as for Theorem~\ref{A1}. Indeed, using \eqref{223} instead of \eqref{1}, inequality  \eqref{27} has to be replaced by
\begin{equation}\label{229}
\bsp
-\wt{w}_t-\wt{w}+A(t,x)\wt{w}_{xx}+A(t,y)\wt{w}_{yy}  +re^{-t} \ge f_1(t,y,u,-v_y)-f_1(t,x,u,v_x), \quad (t,x,y)\in\overline{P}_\tau\setminus \Gamma .
\end{split}
\end{equation}
Assuming that $\wt{w}$ attains a positive maximum at some point $(t_0, \ell ,y_0)\in\overline{P}_\tau\setminus \Gamma$, we deduce (see \eqref{228})
\bqnn
v_x(t_0, \ell ,y_0)=-v_y(t_0, \ell ,y_0)
\eqnn
and so it follows from \eqref{229} and \eqref{225} that
\begin{equation*}
-\wt{w}_t-\wt{w}+A(t,x)\wt{w}_{xx}+A(t,y)\wt{w}_{yy}+re^{-t}\ge 0 \quad\text{ at } \ (t,x,y)=(t_0,x_0,y_0)
\end{equation*}
in contradiction to \eqref{230}, the latter being derived exactly as in Theorem \ref{A1}. Following the proof of Theorem~\ref{A1} one then considers $\wt{w}$ on $S_2$. Replacing $f$ by $f+f_1$ and $g$ by $g+g_1$, we obtain instead of \eqref{41} the inequality
\begin{equation}\label{231}
\bsp
-\wt{w}_t-\wt{w}+ & A(t,y)\wt{w}_{yy}+r_2e^{-t}+r_1e^{-t
}\\
&\ge e^{-t}\big[f_1(t,y,u(t,y),-v_y(t,\ell,y))-g_1(t,\ell,u(t,y),-v_y(t,\ell,y))\big]
\end{split}
\end{equation}
for $(t,\ell,y)\in S_2$.
Assuming that $\wt{w}$ attains a positive maximum at $(t_0,\ell,y_0)\in S_2$ then (see \eqref{233})
\bqnn
v_x(t_0,\ell,y_0) \geq -v_y(t_0,\ell,y_0) \geq q_0
\eqnn
and \eqref{226}, \eqref{231} entail that
\bqnn
-\wt{w}_t-\wt{w}+A(t,y)\wt{w}_{yy}+r_2e^{-t}+r_1e^{-t}\geq 0 \quad \text{ at }\ (t_0,\ell,y_0)\in S_2
\eqnn
contradicting \eqref{234}. The rest follows analogously to the proof of Theorem~\ref{A1}.
\end{proof}
%%%%%%%%%%%%%%%%%%%%%%%%%%%%%%%%%%%%%%%%%%%%%%%%%%%%%%%%%%%%%%%%%%%%%%%%%%%%%%%%%%%%%%%%%%%%%%%%%%%%%%%%%%%%%%%%%%%%%%%%%%%%%%%%%%%%%%%%%%%%%%%%%%%%%%%%%%%%%%%%%%%%%%%%%%%%%%%%%%%%%%%%%%%%%%%%%%
%%%%%%%%%%%%%%%%%%%%%%%%%%%%%%%%%%%%%%%%%%%%%%%%%%%%%%%%%%%%%%%%%%%%%%%%%%%%%%%%%%%%%%%%%%%%%%%%%%%%%%%%%%%%%%%%%%%%%%%%%%%%%%%%%%%%%%%%%%%%%%%%%%%%%%%%%%%%%%%%%%%%%%%%%%%%%%%%%%%%%%%%%%%%%%%%%%
%%%%%%%%%%%%%%%%%%%%%%%%%%%%%%%%%%%%%%%%%%%%%%%%%%%%%%%%%%%%%%%%%%%%%%%%%%%%%%%%%%%%%%%%%%%%%%%%%%%%%%%%%%%%%%%%%%%%%%%%%%%%%%%%%%%%%%%%%%%%%%%%%%%%%%%%%%%%%%%%%%%%%%%%%%%%%%%
%%%%%%%%%%%%%%%%%%%%%%%%%%%%%%%%%%%%%%%%%%%%%%%%%%%%%%%%%%%%%%%%%%%%%%%%%%%%%%%%%%%%%%%%%%%%%%%%%%%%%%%%%%%%%%%%%%%%%%%%%%%%%%%%%%%%%%%%%%%%%%%%%%%%%%%%%%%%%%%%%%%%%%%%%%%%%%%%%%%%%%%%%%%%%%%%%%

\subsubsection{\textbf{Gradient estimates for mixed boundary conditions}}\label{subsection_3}
In \cite[Lemma 1-Lemma 3]{Ter} gradient estimates were derived for solutions to \eqref{1} subject to Dirichlet, Neuman, or Robin type boundary conditions. Inspection of the proofs therein
%\begin{equation}\label{268}
%\bsp
%u_x(t,-\ell) & =u_x(t,\ell)=0,   \\
%u(t,-\ell) & =u(t,\ell) =0,      \\
%u_x(t,-\ell)+\sigma_1(t,-\ell,u) & = u_x(t,\ell)+\sigma_1(t,\ell,u) =0
%\end{sp\ellit}
%\end{equation}
and the one of Theorem~\ref{A1} shows that such estimates can be obtained under the assumptions of Theorem~\ref{A1} when a Dirichlet, Neuman, or Robin type boundary conditions is imposed on one boundary part and the dynamical boundary condition \eqref{2} on the other. 
We formulate the precise result in the exemplary constellation that
$u$ satisfies at $x=-\ell$ the dynamical boundary condition
\bqn\label{235}
u_t - b(t, -\ell,u,u_x)u_{x}=g(t, -\ell,u,u_x), \quad  t\in (0,T),
\eqn
and at $x=\ell$ the homogeneous Dirichlet condition
\bqn\label{236}
u(t,\ell)=0, \quad  t\in (0,T).
\eqn
All other cases can be considered as well.

\begin{theorem}\label{Theorem_12}
Let $M>0$ and suppose \eqref{10}-\eqref{6} and
\begin{align*}
-g(t, - \ell, z, p)  \le b(t, - \ell, z, p)p , \qquad \label{9b}
-g(t,  \ell, z, -p)  \le b(t,  \ell, z, -p)p , 
\end{align*}
for $(t,z)\in (0,T]\times [-M,M]$ and $p\ge q_0$.
Then there exists a constant
$M_1:=M_1(q_0, M,\psi,K)$  such that if $u$ is any  classical solution to \eqref{1}, \eqref{235}, \eqref{236} on $\Om_\tau$ with $\tau\le T$ subject to $u(0,\cdot)=u^0$ satisfying $|u|_{\Om_\tau}\le M$, then $|u_x|_{\Om_\tau}\leq M_1$.
\end{theorem}

\begin{proof}
The proof is the same as for Theorem~\ref{A1} except when considering the behavior of  $\wt{w}$ and $\wt{w}_1$ on $S_2$. That both functions do not attain a positive maximum on $S_2$ can be shown as in \cite[Lemma 3]{Ter}.
\end{proof}

Theorem \ref{Theorem_12} seems to be interesting in connection with the blow up result proven in \cite[Proposition 2.14]{Con}. There the special case
\bqn\label{270}
\bsp
u_t-u_{xx} & =\psi(u_x), \quad\quad t>0, \; 0<x<\ell,\\
u(t,0) & =0, \qquad\quad\phantom{bl,} t \geq 0, \\
u_t(t,\ell)+u_{x}(t,\ell) & =0, \quad\qquad\phantom{bl,} t > 0, \\
u(0,x) & =u^0(x), \qquad\phantom{,} t \geq 0, \\
\end{split}
\eqn
of \eqref{1}, \eqref{235}, \eqref{236} was considered and shown that if  $\psi\in C^2(\R,(0,\infty))$ with
\bqnn
\underset{s\rightarrow \infty}{\lim\,\inf} \psi(s)>0
\eqnn
violates \eqref{266}, then there is $\ell>0$ such that for each $u^0\in C^2([0,\ell])$ with $u^0(0)=0$, the corresponding classical solution to \eqref{270} evolves a gradient singularity in finite time, even if the solution itself stays bounded  (see also \cite[Remark 2.15]{Con}).

Note that assumption \eqref{9bNEU} is satisfied for  \eqref{270}. If the function $\psi$ satisfies \eqref{266}, then \eqref{9} holds for any $q_0>0$ and any bounded solution $u$ to \eqref{270}. Thus, if $u$ is any bounded solution to \eqref{270} subject to some globally Lipschitz continuous initial value $u^0$, then $|u_x|$ does not blow up in finite time according to Theorem~\ref{Theorem_12}. Consequently, Theorem~\ref{Theorem_12} sharpens the statement of \cite[Proposition 2.14]{Con} in the following way: If \eqref{266} is violated and if $u$ is a maximal solution to \eqref{270} on $[0,t^+)\times [-\ell,\ell]$ with $ |u_x|_{\Om_t}{\rightarrow}\infty$ for $t\rightarrow t^+$ as constructed in \cite[Proposition 2.14]{Con}, then
\bqnn
\frac{1}{2}\int\limits_{0}^{\infty}\frac{\rho\,\rd\rho}{\psi(\rho)}\le |u|_{\Omega_{t^+}}\,.
\eqnn

%%%%%%%%%%%%%%%%%%%%%%%%%%%%%%%%%%%%%%%%%%%%%%%%%%%%%%%%%%%%%%%%%%%%%%%%%%%%%%%%%%%%%%%%%%%%%%%%%%%%%%%%%%%%%%%%%%%%%%%%%%%%%%%%%%%%%%%%%%%%%%%%%%%%%%%%%%%%%%%%%%%%%%%%%%%%%%%%%%%%%%%%%%%%%%%%%%
%%%%%%%%%%%%%%%%%%%%%%%%%%%%%%%%%%%%%%%%%%%%%%%%%%%%%%%%%%%%%%%%%%%%%%%%%%%%%%%%%%%%%%%%%%%%%%%%%%%%%%%%%%%%%%%%%%%%%%%%%%%%%%%%%%%%%%%%%%%%%%%%%%%%%%%%%%%%%%%%%%%%%%%%%%%%%%%%%%%%%%%%%%%%%%%%%%

\subsection{H\"older estimates on the gradient and global existence}\label{Sec1}

We next state a result on H\"older estimates on the gradient of a solution to \eqref{1}, \eqref{2}. It turns out that one can apply the existing theory  \cite[Chapter VI]{Lad} rather easily for which the assumptions of Theorem~\ref{A1} have to be strengthened just slightly.

%%%%%%%%%%%%%%%%%%%%%%%%%%%%%%%%%%%%%%%%%%%%%%%%%%%%%%%%%%%%%%%%%%%%%%%%%%%%%%%%%%%%%%%%%%%%%%%%%%%%%%%%%%%%%%%%%%%%%%%%%%%%%%%%%%%%%%%%%%%%%%%%%%%%%%%%%

\begin{corollary}\label{Theorem_15}
Let $u^0\in C^{1+\alpha}([-\ell,\ell])$ for some $\alpha\in (0,1)$ and let \eqref{10}-\eqref{9bNEU} be satisfied with $M>0$. 
There are numbers $\delta:=\delta(M,\alpha)\in (0,\alpha]$ and $C_*:=C_*(M,\psi,\langle u_x^0 \rangle_\Om^\alpha)$ such that if $u\in C^{1,2}(\overline{\Om}_\tau)$ is any classical solution  to \eqref{1}, \eqref{2} on $\Om_\tau$ with $\tau\le T$ subject to $u(0,\cdot)=u^0$ satisfying $\vert u\vert_{\Omega_\tau}\le M$, then
$\langle u_x \rangle^{(\frac{\delta}{2},\delta)}_{\Om_\tau}\leq C_*$.
\end{corollary}

\begin{proof}
Owing to Theorem~\ref{A1} there is a constant $M_1=M_1(M,\psi,K)$ independent of $\tau$ 
such that $|u_x|_{\Om_\tau}\le M_1$.
 Since $a$ and $f$ are continuous and $a$ is positive, it follows  that there is $\mu>0$ such that
\bqnn
\mu^{-1}\le  a(t,x,z,p) \le \mu\, , \quad |f(t,x,z,p)| \le \mu\,, \qquad (t,x,z,p)\in\overline{\Om}_T\times[-M,M]\times[-M_1,M_1].
\eqnn
Moreover, setting
\bqnn
c_1:=\max \big\{ |b(t,\pm \ell, z, p) p|+|g(t,\pm \ell, z, p)| \,;\, (t,z,p)\in [0,T]\times[-M,M]\times[-M_1,M_1]\big\}
\eqnn
we have $\max_{t\in [0,\tau]} |u_t(t,\pm \ell)| \leq c_1$ due to \eqref{2}. 
The assertion is now a consequence of \cite[Chapter VI, Theorem 5.1]{Lad}.
\end{proof}

Summarizing our findings we can simplify the criterion for global existence to problem \eqref{1}, \eqref{2a} from Theorem~\ref{ExLTheorem1}. While the latter requires a uniform bound on the H\"older norms of solutions, the following corollary states that bounds on the supremum norm are sufficient.

\begin{corollary}\label{c3}
Let $\alpha\in (0,1)$ and suppose that \eqref{66a} holds, where $b$ and $g$ are gradient-independent and $a>0$ and $b>0$. Consider $u^0\in C^{2+\alpha}([-\ell,\ell])$ satisfying the compatibility condition \eqref{66}. Suppose there is a constant $M>0$ such that \eqref{10}-\eqref{9bNEU} are satisfied. If the unique solution $u$ to \eqref{1}, \eqref{2a} subject to $u(0,\cdot)=u^0$ on the maximal interval of existence $[0,\tau_\infty)$ provided by Theorem~\ref{ExLTheorem1} satisfies $|u|_{{\Om}_\tau}\leq M$ for $\tau<\tau_\infty$, then $u\in \mC^{1+\frac{\alpha}{2}, 2+\alpha}(\OmT)$.
\end{corollary} 

%%%%%%%%%%%%%%%%%%%%%%%%%%%%%%%%%%%%%%%%%%%%%%%%%%%%%%%%%%%%%%%%%%%%%%%%%%%%%%%%%%%%%%%%%%%%%%%%%%%%%%%%%%%%%%%%%%%%%%%%%%%%%%%%%%%%%%%%%%%%%%%%%%%%%%%%%%%%%%%%%%%%%%%%%%%%%%%%%%%%%%%%%%%%%%%%%%%%%%%%%

\begin{proof}
Let $u$ be the unique solution  to \eqref{1}, \eqref{2a} subject to $u(0,\cdot)=u^0$ on the maximal interval of existence $[0,\tau_\infty)$ provided by Theorem~\ref{ExLTheorem1} and suppose that $|u|_{{\Om}_\tau}\leq M$ for $\tau<\tau_\infty$.
Then, by Theorem~\ref{A1}, Corollary~\ref{Theorem_15} and arguments parallel to those at the end of the proof Theorem~\ref{ExLTheorem1} (cf. \eqref{80a}-\eqref{80e}) we have
$$
|u|^{(\frac{2+\delta}{2},2+\delta)}_{\Om_\tau} \le c(M,\ell,T)
$$
for some $\delta\in (0,\alpha]$ and some constant $c(M,\ell,T)$ independent of $\tau$. By embedding we obtain
$$
\sup_{\tau<\tau_\infty}\big(|u|_{\Om_\tau} + |u_x|_{\Om_\tau} + \la u_x \ra^{(\frac{\alpha}{2}, \alpha)}_{\Om_\tau}\big)<\infty\,.
$$ The assertion follows now from a further application of Theorem~\ref{ExLTheorem1}.
\end{proof}

\begin{remark} For simplicity we stated Corollary~\ref{c3} only for the case of the Bernstein-Nagumo condition considered in Theorem~\ref{A1}. The result is also true in the case of the weakened Bernstein-Nagumo condition considered in Theorem~\ref{Theorem_11}. 
\end{remark}

%%%%%%%%%%%%%%%%%%%%%%%%%%%%%%%%%%%%%%%%%%%%%%%%%%%%%%%%%%%%%%%%%%%%%%%%%%%%%%%%%%%%%%%%%%%%%%%%%%%%%%%%%%%%%%%%%%%%%%%%%%%%%%%%%%%%%%%%%%%%%%%%%%%%%%%%%%%%%%%%%%%%%%%%%%%%%%%%%%%%%%%%%%%%%%%%%%%%%%%%%
\subsection{A priori estimate on the solution}\label{Abschnitt_3}

In the previous subsections we have derived {\it a priori} estimates on the gradient and its H\"older norm of a solution $u$ to \eqref{1}, \eqref{2} in dependence on $|u|_{\Om_\tau}$. As in \cite[Chapt.~I, Thm. 2.9]{Lad}  we shall now also consider a simple situation for which one can derive an {\it a priori} bound on $|u|_{\Om_\tau}$ itself. This yields the bound \eqref{uo} required for global existence.\\

%%%%%%%%%%%%%%%%%%%%%%%%%%%%%%%%%%%%%%%%%%%%%%%%%%%%%%%%%%%%%%%%%%%%%%%%%%%%%%%%%%%%%%%%%%%%%%%%%%%%%%%%%%%%%%%%%%%%%%%%%%%%%%%%%%%%%%%%%%%%%%%%%%%%%%%%%%%%%%%%%%%%%%%%%%%%%%%%%%%%%%%%%%%%%%%%%%%%%%%%%

{Let there be a number $B>0$ such that $f$ and $g$ satisfy the growth conditions 
\bqn\label{209b}
z f(t,x,z,0)\leq \Phi(\vert z\vert) \vert z\vert +B, \quad zg(t,\pm\ell,z,p)\leq \Phi(\vert z\vert) \vert z\vert +B,\quad (t,x,z,p)\in\Om_T\times\R \times \R\,,
\eqn
where $\Phi$ is a non-decreasing positive function on $[0,\infty)$ with
\bqn\label{phi}
\int_0^\infty\frac{\rd r}{\Phi(r)}=\infty .
\eqn
Let  $\phi(\xi)$ be defined for $\xi\in (0,\infty)$ by
\bqnn
\int_0^{\phi(\xi)}\frac{\rd r}{\Phi(r)}=\ln \xi
\eqnn
and note that $\phi$ is monotonically increasing from $-\infty$ to $\infty$ with $\phi(1)=0$ and satisfies $\Phi(\phi(\xi))=\xi\phi'(\xi)$ for $\xi>0$. Then solutions to \eqref{1}, \eqref{2} are bounded:

\begin{proposition}\label{Theorem_14b}
Suppose \eqref{209b} with \eqref{phi} and set
\bqn\label{M}
M:= \inf\left\{ \phi(\xi)\,:\, \lambda >1\,,\,\xi= \max\left\{  1\,,\, \phi^{-1}\left(\frac{B}{(\lambda-1)\Phi(0)}\right) \,,\, \phi^{-1}\big( |u^0|_{[-\ell,\ell]}\big)\right\}\right\}.
\eqn
 If $u$ is any classical solution to \eqref{1}, \eqref{2} on $\Om_\tau$ with $\tau\le T$ subject to $u(0,\cdot)=u^0$, then
%\bqnn\label{213}
$|u|_{\Om_\tau} \leq M$.
%\eqnn
\end{proposition}
%%%%%%%%%%%%%%%%%%%%%%%%%%%%%%%%%%%%%%%%%%%%%%%%%%%%%%%%%%%%%%%%%%%%%%%%%%%%%%%%%%%%%%%%%%%%%%%%%%%%%%%%%%%%%%%%%%%%%%%%%%%%%%%%%%%%%%%%%%%%%%%%%%%%%%%%%%%%%%%%%%%%%%%%%%%%%%%%%%%%%%%%%%%%%%%%%%%%%%%%%

\begin{proof}
The proof follows along the lines of the proof of \cite[Chapt. I, Thm. 2.9]{Lad} and differs merely where the dynamical boundary conditions come into play. For the reader's ease we give the complete proof here. 
Put $u=\phi(v)$ and, given $\lambda>1$, set  
\bqnn
\hat v(t,x):=v(t,x)e^{-\lambda t}, \quad (t,x)\in\overline{\Om}_\tau .
\eqnn
Let $\hat v$ attain its maximum at some point $(t_0,x_0)\in\ol{\Om}_\tau$. We want to derive an upper bound on $\hat v$ and may thus restrict to the case $\hat v(t_0,x_0)\ge e^{-\lambda t_0}$, that is, $v(t_0,x_0)\ge 1$ and $u(t_0,x_0)\ge 0$. Let
\bqnn
\Gamma:=\{0\}\times [-\ell,\ell] \cup (0,\tau)\times \{\pm\ell\}
\eqnn
be the parabolic boundary of $\Om_\tau$.

{\bf Case 1:} Assume $(t_0,x_0)\in \ol{\Om}_\tau\setminus \Gamma$. In this case, $(t_0,x_0)$ does not belong to $\Gamma$ and at this point we have 
\bqn\label{214b}
\hat v(t_0,x_0)>0,\quad \hat v_x(t_0,x_0)=0,\quad \hat v_{xx}(t_0,x_0)\leq 0 ,\quad \hat v_{t}(t_0,x_0)\geq 0.
\eqn
According to \eqref{1}, 
\bqnn\label{215}
\hat v_te^{\lambda t}+\lambda \hat ve^{\lambda t}-a(t,x,u,u_x)\left(\hat v_{xx}e^{\lambda t}+\frac{\phi''(v)}{\phi'(v)}e^{2\lambda t}\hat v_x^2\right) =\frac{1}{\phi'(v)} f\big(t,x,u,u_x\big)
\eqnn
and thus, by \eqref{214b},
\bqnn
\lambda \hat v(t_0,x_0)e^{\lambda t_0} \leq \frac{1}{\phi'(v)} f\big(t_0,x_0,u,0\big).
\eqnn
Multiplying the previous inequality by $\phi'(v(t_0,x_0))u(t_0,x_0) \ge 0$ and invoking \eqref{209b} yields
$$
\lambda \phi'(v) v u \le f\big(t_0,x_0,u,0\big) u \le B+ \Phi(u) u\quad \text{at }\ (t,x)=(t_0,x_0)\,,
$$
which also reads
$$
(\lambda-1) \Phi(u) u \le B\quad \text{at }\ (t,x)=(t_0,x_0).
$$
Hence, since $\Phi$ is non-decreasing,
\bqnn
u(t_0,x_0) \leq \frac{B}{(\lambda-1)\Phi(0)} 
\eqnn
from which we deduce, since $\phi^{-1}$ is increasing,
\bqn\label{uuu}
v(t,x)\le e^{\lambda t} \hat v(t_0,x_0) =\phi^{-1}\big(u(t_0,x_0)\big) e^{\lambda (t-t_0)}\leq \phi^{-1}\left(\frac{B}{(\lambda-1)\Phi(0)}\right)  e^{\lambda T}, \qquad (t,x)\in \overline\Om_\tau . 
\eqn

{\bf Case 2:} Assume $(t_0,x_0)\in \Gamma$.
Suppose first that $t_0=0$. Then
\bqn\label{218b}
v(t,x)\le e^{\lambda t} \hat v(t_0,x_0) =e^{\lambda t} \phi^{-1}\big(u(0,x_0))\leq e^{\lambda T} \phi^{-1}\big( |u^0|_{[-\ell,\ell]}\big), \qquad (t,x)\in \overline\Om_\tau.
\eqn
If $t_0\neq 0$, then $(t_0,x_0)\in  (0,\tau)\times \{\pm \ell\}$ in which case we may use the boundary condition \eqref{2}. So,  if $(t_0,x_0)\in (0,\tau)\times \{\ell\}$, then  $b(\cdot,\ell,\cdot,\cdot)>0$ and the fact that $(t_0,\ell)$ is a point of a positive maximum imply
\bqnn
 v(t_0,\ell)>0,\quad v_{t}(t_0,\ell)= 0, \quad b\big(t_0,\ell, u(t_0,\ell), u_x(t_0,\ell)\big)v_{x}(t_0,\ell)\geq 0. 
\eqnn
From \eqref{2} and \eqref{209b} we deduce, at $(t,x)=(t_0,\ell)$,
\bqnn
\bsp
\lambda \phi'(v) \hat v e^{\lambda t} & \leq  e^{\lambda t} \phi'(v)\big(\lambda \hat v  + \hat v_t \big) + b(t,x,u,u_x) \phi'(v) e^{\lambda t} \hat v_x = g\big(t, \ell, u,u_x\big) \le B+ \Phi(u) u,
\end{split}
\eqnn
which again reads
$$
(\lambda-1) \Phi(u) u \le B\quad \text{at }\ (t,x)=(t_0,\ell)
$$
and we may proceed as in Case 1 to derive \eqref{uuu}. Clearly, the same argument holds when $(t_0,x_0)\in (0,\tau)\times \{-\ell\}$.

Combining the different cases from \eqref{uuu} and \eqref{218b} we derive that
\bqnn
\max_{\ol{\Om}_\tau}u\leq  M
\eqnn
with $M$ given in \eqref{M}. Considering $-u$ on $\ol{\Om}_\tau$ yields an estimate from below on $u$. This proves the assertion.
\end{proof}
%%%%%%%%%%%%%%%%%%%%%%%%%%%%%%%%%%%%%%%%%%%%%%%%%%%%%%%%%%%%%%%%%%%%%%%%%%%%%%%%%%%%%%%%%%%%%%%%%%%%%%%%%%%%%%%%%%%%%%%%%%%%%%%%%%%%%%%%%%%%%%%%

\begin{corollary}\label{C33}
Let $\alpha\in (0,1)$ and suppose \eqref{66a}. Consider $u^0\in  C^{2+\alpha}([-\ell,\ell])$ satisfying the compatibility condition \eqref{66}. Suppose \eqref{209b} with \eqref{phi} and that
 \eqref{10}-\eqref{9bNEU} are satisfied with $M$ given in \eqref{M}. Then there exists a global solution $u\in \mC^{1+\frac{\alpha}{2}, 2+\alpha}(\OmT)$ to \eqref{1}, \eqref{2a} subject to $u(0,\cdot)=u^0$.
\end{corollary}

%%%%%%%%%%%%%%%%%%%%%%%%%%%%%%%%%%%%%%%%%%%%%%%%%%%%%%%%%%%%%%%%%%%%%%%%%%%%%%%%%%%%%%%%%%%%%%%%%%%%%%%%%%%%%%%%%%%%%%%%%%%%%%%%%%%%%%%%%%%%%%%%%%%%%%%%%%%%%%%%%%%%%%%%%%%%%%%%%%%%%%%%%%%%%%%%%%%%%%%%%
\begin{proof}
This now follows from Theorem~\ref{ExLTheorem1}, Theorem~\ref{A1}, Corollary~\ref{Theorem_15}, and Proposition~\ref{Theorem_14b}.
\end{proof}

%%%%%%%%%%%%%%%%%%%%%%%%%%%%%%%%%%%%%%%%%%%%%%%%%%%%%%%%%%%%%%%%%%%%%%%%%%%%%%%%%%%%%%%%%%%%%%%%%%%%%%%%%%%%%%%%%%%%%%%%%%%%%%%%%%%%%%%%%%%%%%%%%%%%%%%%%%%%%%%%%%%%%%%%%%%%%%%%%%%%%%%%%%%%%%%%%%%%%%%%%

\begin{remark}
For simplicity we stated the global existence result of Corollary~\ref{C33} only for the case of the Bernstein-Nagumo condition considered in Theorem~\ref{A1}. The result is also true in the case of the weakened Bernstein-Nagumo condition considered in Theorem~\ref{Theorem_11}.
\end{remark}

As a final remark let us point out that, as in \cite{Ter}, we could replace \eqref{1} by the more general equation
\bqnn
u_t=F(t,x,u,u_x,u_{xx})
\eqnn
and derive similar results provided that $F$ is differentiable with respect to its last variable.

%%%%%%%%%%%%%%%%%%%%%%%%%%%%%%%%%%%%%%%%%%%%%%%%%%%%%%%%
%%%%%%%%%%%%%%%%%%%%%%%%%%%%%%%%%%%%%%%%%%%%%%%%%%%%%%%%
\section{Appendix: Notation and Parabolic H\"older Spaces}\label{app}
%%%%%%%%%%%%%%%%%%%%%%%%%%%%%%%%%%%%%%%%%%%%%%%%%%%%%%%%
%%%%%%%%%%%%%%%%%%%%%%%%%%%%%%%%%%%%%%%%%%%%%%%%%%%%%%%%

 We provide in this appendix the definition of parabolic H\"older spaces, for more details we refer e.g. to \cite{Lad, Lie}.
 Recall that $\Omega =(-\ell,\ell)$ and $\Omega_T =(0,T)\times \Om$. Let $Q$ be an interval in $\R$ and $\alpha\in (0,1]$. Then $u:Q\rightarrow \R$ is uniformly $\alpha$-H\"older continuous if 
\bqnn
[u]^{(\alpha)}_Q:=\us{x,y\in Q}{\sup} \frac{|u(x)-u(y)|}{|x-y|^\alpha} <\infty.
\eqnn
We say that $u:Q\rightarrow \R$ is $\alpha$-H\"older continuous if every point in $Q$ has a neighborhood $Q'\subset Q$ such that $u\vert_{Q'}$ is uniformly $\alpha$-H\"older continuous. Note that $\alpha$-H\"older continuous functions are uniformly $\alpha$-H\"older continuous on compact sets. 
%If $\alpha\in (0,1)$ we let the set of  $\alpha$-H\"older continuous functions denote by $C^\alpha(Q)$ while $\mathcal{C}^\alpha(Q)$ is the set of uniformly $\alpha$-H\"older continuous functions. If $Q$ is bounded, then $C^\alpha(\overline Q)=\mathcal{C}^\alpha(Q)$.
Clearly, $1$-H\"older continuous functions are Lipschitz continuous. %We put%
%$$
%C^{1-}(Q):=\{ u:Q\rightarrow \R : u \text{ is Lipschitz continuous}\}.
%$$
Given $k\in \N$ and $\alpha\in (0,1)$ we let $C^{k+\alpha}(Q)$ be the space of all $k$-times continuously differentiable functions on $Q$ such that the $k$-th derivative is $\alpha$-H\"older continuous.
%with finite norm
%\bqnn
%|u|^{(k+\alpha)}_{Q}:=\sml_{j\leq k}|D^j u|_{Q}+[D^ku]^{(\alpha)}_Q,
%\eqnn
%where
If $u:S\rightarrow\R$ is bounded on a set $S$, we denote its supremum norm by $|u|_{S}$, i.e.
\bqnn
|u|_{S}:=	\us{x\in S}{\sup}|u(x)|.
\eqnn
%and the convention that $[u]^{(\alpha)}_Q:=0$ if $\alpha=0$.
If $Q$ is bounded and $k\in \N$, $\alpha\in (0,1)$ given, we put 
\bqnn
|u|^{(k+\alpha)}_{Q}:=\sml_{j\leq k}|D^j u|_{Q}+[D^k u]^{(\alpha)}_Q
\eqnn
for $u\in C^{k+\alpha}(\bar Q)$.

Let $T>0$ and $Q_T=(0,T)\times Q$. We let $C^{1,2}(Q_T)$ be the space of all continuous real-valued functions defined on $Q_T$ having continuous derivatives $u_t, u_{xx}$ on $Q_T$ and we use a similar notation on $\overline Q_T$. 
%If $\alpha\in (0,1)$ we mean by writing $a\in C^{\frac{\alpha}{2},\alpha,1-,1-}([0,T]\times[-\ell,\ell]\times\R\times\R)$ that for each $N>0$ there is $c_N>0$ such that
%$$
%|a(t_1,x_1,z_1,p_1)-a(t_2,x_2,z_2,p_2)|\leq  c_N %(|t_1-t_2|^{\Alphah}+|x_1-x_2|^\alpha+|z_1-z_2|+|p_1-p_2|)%
%$$
%for $(t_i,x_i,z_i,p_i)\in [0,T]\times [-\ell,\ell]\times [-N,N]\times [-N,N]$ and similarly if less variables are involved.

The parabolic H\"older spaces
  $\mC^{\frac{k+\alpha}{2}, k+\alpha}(Q_T)$ for $k\in \N$ and $\alpha\in (0,1)$ are defined as the set of functions $u : Q_T\rightarrow \R$ having finite norm %{\color{red} delete: (see also \cite[p.46]{Lie})}
	\bqnn\label{118}
|u|^{(\frac{k+\alpha}{2}, k+\alpha)}_{Q_T}:=\sml_{2j_0+j\leq k}|D_t^{j_0}D_x^j u|_{Q_T}+\la u \ra^{(\frac{k+\alpha}{2}, \,k+\alpha)}_{Q_T} < \infty,
\eqnn
where, if $k=0$,
\bqnn
\la u \ra^{(\frac{\alpha}{2}, \alpha)}_{Q_T}  :=\la u \ra^{(\frac{\alpha}{2})}_{t,Q_T}+\la u \ra^{(\alpha)}_{x,Q_T}
\eqnn
with
\bqnn
\la u \ra^{(\alpha )}_{t,Q_T}:=\us{x\in Q}{\sup}[u(\cdot,x)]^{( \alpha)}_{(0,T)}, \qquad 
\quad \la u \ra^{(\alpha)}_{x,Q_T}:=\us{t\in (0,T)}{\sup}[u(t, \cdot)]^{(\alpha)}_{\Om},
\eqnn
and, if $k \geq 1$,
\bqn\label{86}
\bsp
\la u \ra^{(\frac{k+\alpha}{2},\, k+\alpha)}_{Q_T}  &:=\sml_{2j_0+j=k}\la D_t^{j_0}D_x^j u\ra_{Q_T}^{(\frac{\alpha}{2},\alpha)} \; + \;  \sml_{2j_0+j=k-1}\la D_t^{j_0}D_x^j u\ra_{t,Q_T}^{(\frac{1+\alpha}{2})}  \,.
%{\tred +                              \sml_{2j_0+j=k-1}\la D_t^{j_0}D_x^j u\ra_{t,Q_T}^{(\frac{1+\alpha}{2})}}.
\end{split}
\eqn
%{\tred Let us point out that the inclusion of the last term on the right hand side of \eqref{86} is handy, but superfluous as it may be estimated by the first one and $\vert u\vert_{Q_T}$ (see \cite[p.46]{Lie} and\cite[Exercise 8.8.6]{Kry}).}\footnote{{\tred Koennen wir dies nicht direkt weglassen hier (vgl. Anfang von Sect.4)??}}
Since globally H\"older continuous functions are uniformly continuous, it follows that
\bqnn
D_t^{j_0}D_x^j u\in C(\overline{Q}_T), \quad 2j_0+j\leq k,
\eqnn
for $u\in \mC^{\frac{k+\alpha}{2}, k+\alpha}(Q_T)$.

The following interpolation inequalities are quite useful. We refer to \cite[Thm.8.8.1]{Kry} (see also \cite[Exercise 8.8.2]{Kry}) for a proof.

\begin{proposition}\label{FR1}
There is a constant $c=c(Q)$ such that for $u\in \mC^{1+\frac{\alpha}{2},2+\alpha}(Q_T)$:
\bqn\label{71}
\bsp
|u_t|_{Q_T}+|u_{xx}|_{Q_T}&\leq c\,U_{2+\alpha}^{1/(2+\alpha)}U_{0}^{1-1/(2+\alpha)} ,\\
\la u_x \ra_{Q_T}^{(\frac{\alpha}{2},\alpha)}&\leq c\,U_{2+\alpha}^{(1+\alpha)/(2+\alpha)}U_{0}^{1-(1+\alpha)/(2+\alpha)} ,\\
\la u \ra_{Q_T}^{(\frac{\alpha}{2},\alpha)}&\leq c\,U_{2+\alpha}^{\alpha/(2+\alpha)}U_{0}^{1-\alpha/(2+\alpha)} ,\\
\end{split}
\eqn
where $U_0:=|u|_{Q_T}$ and $U_{2+\alpha}:=\la u \ra^{(1+\frac{\alpha}{2},2+\alpha)}_{Q_T}$.
\end{proposition}

\end{document}